\documentclass[11pt]{amsart}
\usepackage{mathtools,amsmath,amssymb}
\usepackage{amsthm}
\usepackage{graphicx}
\usepackage[dvipsnames]{xcolor}
\usepackage{tikz-cd}
\usepackage[OT2,T1]{fontenc}
\usepackage{bbm}
\usepackage{stmaryrd}
\usepackage
{hyperref}
\usepackage[backend=biber,style=alphabetic]{biblatex}
\renewbibmacro{in:}{}

\mathtoolsset{showonlyrefs,showmanualtags}

\DeclareFieldFormat{postnote}{#1}
\DeclareFieldFormat{multipostnote}{#1}

\DeclareFieldFormat[article]{citetitle}{\mkbibemph{#1}}
\DeclareFieldFormat[article]{title}{\mkbibemph{#1}}
\DeclareFieldFormat[article]{volume}{\mkbibbold{#1}}
\DeclareFieldFormat[article]{journaltitle}{#1}

\DeclareFieldFormat[inbook]{citetitle}{\mkbibemph{#1}}
\DeclareFieldFormat[inbook]{title}{\mkbibemph{#1}}

\DeclareFieldFormat[inproceedings]{citetitle}{\mkbibemph{#1}}
\DeclareFieldFormat[inproceedings]{title}{\mkbibemph{#1}}
\DeclareFieldFormat[inproceedings]{booktitle}{#1}

\DeclareFieldFormat[incollection]{citetitle}{\mkbibemph{#1}}
\DeclareFieldFormat[incollection]{title}{\mkbibemph{#1}}
\DeclareFieldFormat[incollection]{booktitle}{#1}

\hypersetup{colorlinks=true,citecolor=blue,linkcolor=blue,urlcolor=blue,
pdfstartview=FitH }
\allowdisplaybreaks

\theoremstyle{plain}
\newtheorem{theorem}{Theorem}
\numberwithin{theorem}{section}

\newtheorem{lemma}[theorem]{Lemma}
\newtheorem{prop}[theorem]{Proposition}

\theoremstyle{definition}

\newtheorem{remark}[theorem]{Remark}
\newtheorem*{remark*}{Remark}

\newtheorem{example}[theorem]{Example}

\newtheoremstyle{named}{}{}{\itshape}{}{\bfseries}{.}{.5em}{\thmnote{#3's }#1}
\theoremstyle{named}

\DeclarePairedDelimiter\abs{\lvert}{\rvert}%

\DeclareMathOperator{\diag}{diag}

\DeclareMathOperator{\id}{id}
\DeclareMathOperator{\im}{Im}

\DeclareMathOperator{\kl}{Kl}
\DeclareMathOperator{\mo}{mod}

\def\C{\mathbb{C}}

\def\F{\mathbb{F}}
\def\G{\mathbb{G}}

\def\Q{\mathbb{Q}}

\def\Z{\mathbb{Z}}

\title{Kloosterman sums on orthogonal groups}
\author{Catinca Mujdei}
\thanks{\textit{Affiliation}. University College London, London, UK}
\thanks{\textit{E-mail address}. \url{catinca.mujdei.23@ucl.ac.uk}}
\thanks{The author was supported by the Engineering and Physical Sciences Research Council [EP/S021590/1], via the EPSRC Centre for Doctoral Training in Geometry and Number Theory (The London School of Geometry and Number Theory), University College London.}

\makeatletter
\def\Ddots{\mathinner{\mkern1mu\raise\p@
\vbox{\kern7\p@\hbox{.}}\mkern2mu
\raise4\p@\hbox{.}\mkern2mu\raise7\p@\hbox{.}\mkern1mu}}
\makeatother
\makeatletter
\DeclareRobustCommand\widecheck[1]{{\mathpalette\@widecheck{#1}}}
\def\@widecheck#1#2{%
    \setbox\z@\hbox{\m@th$#1#2$}%
    \setbox\tw@\hbox{\m@th$#1%
       \widehat{%
          \vrule\@width\z@\@height\ht\z@
          \vrule\@height\z@\@width\wd\z@}$}%
    \dp\tw@-\ht\z@
    \@tempdima\ht\z@ \advance\@tempdima2\ht\tw@ \divide\@tempdima\thr@@
    \setbox\tw@\hbox{%
       \raise\@tempdima\hbox{\scalebox{1}[-1]{\lower\@tempdima\box
\tw@}}}%
    {\ooalign{\box\tw@ \cr \box\z@}}}
\makeatother

\begin{document}

\begin{abstract}
We study Kloosterman sums on the orthogonal groups $SO_{3,3}$ and $SO_{4,2}$, associated to short elements of their respective Weyl groups. An explicit description for these sums is obtained in terms of multi-dimensional exponential sums. These are bounded by a combination of methods from algebraic geometry and $p$-adic analysis.
\end{abstract}

\maketitle

\section{Introduction}
The classical Kloosterman sum is defined as 
\begin{equation}\label{eq:classic}
S(m,n;c)= \sum_{\substack{d\,(\mo c) \\ (c,d)=1}}e\left(\frac{m\overline{d}+nd}{c}\right),
\end{equation}
for parameters $m,n\in\Z$ and modulus $c\in\Z_{>0}$. Here we use the notation $e(x)=\exp(2\pi ix)$, and $d\overline{d}\equiv 1\,(\mo c)$ if $(c,d)=1$. These exponential sums bear the name of Hendrik Kloosterman, who encountered them in his study of quaternary quadratic forms by a refinement of the circle method \cite{K}. Weil \cite{W} later observed a connection between classical Kloosterman sums and the Riemann hypothesis for curves over finite fields. He proved that
\begin{equation}\label{eq:weil_bound}
\abs{S(m,n;c)} \leq \tau(c)c^{1/2}(m,n,c)^{1/2}
\end{equation}
for $c$ square-free (complemented by the same bound for a prime power $c=p^k$, $k\geq2$, due to Sali\'e \cite{Sa}). This bound is optimal, and it essentially states that Kloosterman sums exhibit square-root cancellation.

Classical Kloosterman sums are ubiquitous in number theory. In the analytic theory of automorphic forms they appear in the Fourier expansion of automorphic kernels. Therefore, having a good understanding of Kloosterman sums is necessary when analyzing the geometric side of the relative trace formula, as motivated in \cite{B}. Higher rank relative trace formulae will require an understanding of the new types of Kloosterman sums that appear in it.

We consider generalized Kloosterman sums as in \cite{rd2}. To define these, consider a split reductive group $G$ over $\Q_p$ with maximal torus $T$. Write $U$ for the unipotent subgroup corresponding to a fixed system of positive roots of $(G,T)$, and $U^-$ for the unipotent subgroup corresponding to the reflections of these positive roots. Consider the normalizer $N= N_G(T)$, the Weyl group $W= N_G(T)/T$, and the associated quotient map $N\rightarrow W:n\mapsto\overline{n}$. For $n\in N$, write $U^{\overline{n}}= U\cap n^{-1}U^-n$, noting that this group only depends on the image of $n$ in $W$. The Bruhat decomposition of $G$ provides a unique product decomposition for each element in $G(\Q_p)$, see e.g. \cite[21.80]{jsm2}. We continue to work over a fixed cell in this decomposition: for $n\in N(\Q_p)$, define the corresponding Kloosterman set
\begin{equation}\label{eq:kloosterman_set}
X(n)= U(\Z_p)\backslash (U(\Q_p)nU^{\overline{n}}(\Q_p)\cap G(\Z_p))/U^{\overline{n}}(\Z_p),
\end{equation}
along with the projection maps
$$u:X(n)\rightarrow U(\Z_p)\backslash U(\Q_p)\quad\text{ and }\quad u':X(n)\rightarrow U^{\overline{n}}(\Q_p)/U^{\overline{n}}(\Z_p)$$
such that $x=u(x)nu'(x)$ for $x\in X(n)$ arbitrary. Given two characters $\psi,\psi':U(\Q_p)\rightarrow \C^\times$ that are trivial on $U(\Z_p)$, the (local) Kloosterman sum is then defined to be
\begin{equation}\label{local}
\kl_p(\psi,\psi';n)=\sum_{x\in X(n)}\psi(u(x))\psi'(u'(x)).
\end{equation}

\begin{example}\label{ex:classical_gl2}
For $G= GL_2$, take $T=\left\{\begin{pmatrix}
t &  \\  & t^{-1}
\end{pmatrix}\,:\,t\in\Q_p^\times \right\}$ and $U= \left\{\begin{pmatrix}
1 & x \\  & 1
\end{pmatrix}\,:\,x\in\Q_p\right\}$. Choosing the normalizing element 
$$n_r=\begin{pmatrix}
& -p^{-r} \\ p^r & 
\end{pmatrix}$$
for some $r\in\Z_{\geq0}$, we have $U^{\overline{n_r}}=U$. Consider the two characters 
$$\psi\begin{pmatrix}
1 & x \\ 0 & 1
\end{pmatrix}= e(mx),\quad\psi'\begin{pmatrix}
1 & x \\ 0 & 1
\end{pmatrix}= e(m'x),$$
for some $m,m'\in\Z$. Given these data, an element in $X(n_r)$ has a coset representative of the form 
$$\begin{pmatrix}
1 & x \\ 0 & 1
\end{pmatrix} n_r\begin{pmatrix}
    1 & y \\ 0 & 1
\end{pmatrix} = \begin{pmatrix}
    xp^r & -p^{-r}+xyp^r \\
    p^r & yp^r
\end{pmatrix}\in G(\Z_p),\quad x,y\in\Q_p/\Z_p,$$
if and only if $x,y\in p^{-r}\Z_p/\Z_p$ and $-p^{-2r}+xy\in p^{-r}\Z_p/\Z_p$. Filling in this information in \eqref{local} gives $\kl_p(\psi,\psi';n_r)=S(m,m';p^r)$.
\end{example}

\begin{remark}
It is straightforward to define local Kloosterman sums more generally for reductive groups $G$ that are non-split over $\Q_p$ (e.g.\ $SO_{4,2}$ for $p\not\equiv 1\,(\mo 4)$) using relative notions. More precisely, the $\Q_p$-split maximal torus $T$ is replaced by a maximal $\Q_p$-split torus $S$, the Weyl group $W$ is replaced by the relative Weyl group ${}_{\Q_p}W$, and the Bruhat decomposition takes on the relative form in \cite[25.27]{jsm2}.
\end{remark}

The size of the Kloosterman set $X(n)$, and thus the trivial bound for a generalized Kloosterman sum, was determined by D{\k a}browski and Reeder \cite{DR}. It remains an open problem to obtain explicit descriptions and non-trivial bounds for arbitrary generalized Kloosterman sums, though some specific cases have already been established. Friedberg \cite{F} considered Kloosterman sums on $GL_n$ for the Weyl element $\begin{psmallmatrix}
    & 1 \\ I_{n-1} &
\end{psmallmatrix}$, to which he applied Deligne's bound for hyper-Kloosterman sums. Non-trivial bounds have also been obtained for Kloosterman sums on $GL_3$ in \cite{S,BFG,DF}, $Sp_4$ in \cite{M}, $GL_4$ for the long Weyl element in \cite{GSW}, and $GL_n$ for the long Weyl element and 
$$w=\begin{psmallmatrix}
    && \pm1 \\
    & -I_{n-2} & \\ 
    1 &&
\end{psmallmatrix}$$
in \cite{BM}.

In this paper, we provide an explicit description of Kloosterman sums on the special orthogonal group
$$SO_{r,s}=\{A\in GL_{r+s}\,:\,AMA^t=M,\,\det(A)=1\},\quad M=\begin{pmatrix}
    && I_s \\
    & I_{r-s} & \\
    I_s &&
\end{pmatrix},$$
for the low-dimensional cases $(r,s)\in\{(3,3),(4,2)\}$. To this end, we apply Friedberg's results on Pl\"ucker coordinates \cite{F} to characterize the elements of $X(n)$. Equipped with these explicit descriptions, we bound the resulting sums. We restrict to Weyl elements of short length as the Plücker coordinates quickly become more complex as the length increases.

To motivate our specific choices for $r$ and $s$ in $SO_{r,s}$, we note that Kloosterman sums on orthogonal groups have so far only been examined for the particular case $SO_{r,1}$, $r\geq0$, in \cite[Sections 4-6]{je} and \cite[Section 4]{jc}. In fact, Corollary 6.15 of the former states that the resulting Kloosterman sums are again classical. Moreover, as can be easily verified with the same methods as in the current paper, the Kloosterman sum associated to the long Weyl element of $SO_{2,2}$ is the product of two classical Kloosterman sums, while the sums for the Weyl elements of length $1$ are classical. This reflects the fact that there is a $2$-to-$1$ homomorphism $SL_2(\Q_p)\times SL_2(\Q_p)\rightarrow SO_{2,2}(\Q_p)$.

For Kloosterman sums on $SO_{3,3}$ we obtain the following explicit expressions, with notation as introduced in Section \ref{sec:so33_sums_section}:
\begin{theorem}
Let $\psi=\psi_{m_1,m_2,m_3}$, $\psi'=\psi_{n_1,n_2,n_3}$ be characters of $U(\Q_p)/U(\Z_p)$, and $r,s\in\Z_{\geq0}$. Then
\begin{alignat*}{3}
\kl_p(\psi,\psi';n_{\id}) &= 1, \\
\kl_p(\psi,\psi';n_{s_\alpha,r}) &= S(m_2,n_2;p^r), \\
\kl_p(\psi,\psi';n_{s_\beta,r}) &= S(m_1,n_1;p^r), \\
\kl_p(\psi,\psi';n_{s_\gamma,r}) &= S(m_3,n_3;p^r), \\
\kl_p(\psi,\psi';n_{s_\alpha s_\beta,r,s}) &= S_3(-n_1,m_1,m_2;p^r,p^s), &&\quad\quad\quad\text{ if }r\leq s, \\
\kl_p(\psi,\psi';n_{s_\alpha s_\gamma,r,s}) &= S(m_2,n_2;p^s)S(m_3,n_3;p^r), \\
\kl_p(\psi,\psi';n_{s_\beta s_\alpha,r,s}) &= S_3(-n_2,m_2,m_1;p^s,p^r), &&\quad\quad\quad\text{ if }r\geq s, \\
\kl_p(\psi,\psi';n_{s_\beta s_\gamma,r,s}) &= S_3(n_3,m_3,m_1;p^s,p^r), &&\quad\quad\quad\text{ if }r\geq s,\\
\kl_p(\psi,\psi';n_{s_\gamma s_\beta,r,s}) &= S_3(-n_1,m_1,m_3;p^r,p^s), &&\quad\quad\quad\text{ if }r\leq s,
\end{alignat*}
where
\begin{equation}\label{eq:gl3_sum}
S_3(m_1,n_1,n_2;p^r,p^s)=\sum_{\substack{x\,(\mo p^r) \\ (x,p^r)=1}}\sum_{\substack{y\,(\mo p^s) \\ (y,p^{s-r})=1}}e\left(\frac{m_1x+n_1\overline{x}y}{p^r}\right)e\left(\frac{n_2\overline{y}}{p^{s-r}}\right),
\end{equation}
with $m_1,n_1,n_2\in\Z_p$, $r,s\in\Z_{\geq0}$, $r\leq s$, $x\overline{x}\equiv 1\,(\mo p^r)$, $y\overline{y}\equiv 1\,(\mo p^{s-r})$.
\end{theorem}

The $GL_3$ sums \eqref{eq:gl3_sum} were first observed in \cite{BFG}, and were bounded by Larsen:
\begin{theorem}[{\cite[Appendix]{BFG}}]\label{thm:bfg}
For $m_1,n_1,n_2\in\Z_p$, $r,s\in\Z_{\geq0}$ with $r\leq s$, one has
\begin{equation}\label{eq:larsen_bound}
S_3(m_1,n_1,n_2;p^r,p^s) \ll\left(p^{s+\min\{\nu_p(m_1),\nu_p(n_1),r\}},p^{2r+\min\{\nu_p(n_2),s-r\}}\right).
\end{equation}
\end{theorem}
\vspace{-3mm}
Here $\nu_p:\Q_p\rightarrow\Z\cup\{\infty\}$ denotes the $p$-adic valuation.

Regarding $SO_{4,2}/\Q_p$, which is non-split if and only if $p\not\equiv 1\,(\mo 4)$, we only compute the corresponding Kloosterman sums for those primes. We obtain the folowing explicit expressions, with notation as introduced in Section \ref{sec:so42_sums_section}:
\begin{theorem}
Let $\psi=\psi_{m_1,m_2,m_3}$, $\psi'=\psi_{n_1,n_2,n_3}$ be characters of $U(\Q_p)/U(\Z_p)$, and $r\in\Z_{\geq0}$. Then
\begin{align*}
\kl_p(\psi,\psi';n_{\id}) &= 1, \\
\kl_p(\psi,\psi';n_{s_\alpha,r}) &= S_4(m_2,m_3,n_2,n_3;p^r), \\
\kl_p(\psi,\psi';n_{s_\beta,r}) &= S(m_1,n_1;p^r),
\end{align*}
where 
\begin{align}
&S_4(m_2,m_3,n_2,n_3;p^r) \\
&\quad= 
\begin{cases}
\sum\limits_{\substack{x,y\,(\mo p^{r/2}) \\ (x,y,p)=1}}e\left(-2\frac{(m_2x-m_3y)(\overline{x^2+y^2})}{p^{r/2}}\right)e\left(\frac{n_2x+n_3y}{p^{r/2}}\right) &\text{ if }p\equiv3\,(\mo 4), r\text{ is even}, \\
\sum\limits_{\substack{x,y\,(\mo p^{(r-1)/2}) \\ (x,y,p)=1 \\ p\vert xy}}e\left(-\frac{(m_2x-m_3y)(\overline{x^2+y^2})}{p^{(r-1)/2}}\right)e\left(\frac{n_2x+n_3y}{p^{(r-1)/2}}\right) &\text{ if }p=2, r\geq3\text{ is odd}, \\
\sum\limits_{\substack{x,y\,(\mo p^{r/2}) \\ (xy,p)=1}}e\left(-\frac{(m_2x-m_3y)\overline{(x^2+y^2)p^{-1}}}{p^{r/2}}\right)e\left(\frac{n_2x+n_3y}{p^{r/2}}\right) &\text{ if }p=2, r\text{ is even}, \\
 \\
0 &\text{ otherwise}
\end{cases}. \label{eq:so42_sum}
\end{align}
\end{theorem}

Note that for the case $p=2$ and $r\geq0$ even in \eqref{eq:so42_sum}, the values of $x^2$ and $y^2$ are well-defined modulo $p^{r/2+1}$, and $\nu_p(x^2+y^2)=1$, so the value of $(x^2+y^2)p^{-1}$ is well-defined and invertible modulo $p^{r/2}$. 

By a $p$-adic stationary phase method and a result of Hooley \cite{H}, we obtain the following bound:
\begin{theorem}\label{thm:so42_bounds}
Consider the following:
\begin{itemize}
\item[-] If $p\equiv3\,(\mo 4)$ and $r>2$, or if $p=2$ and $r>3$, assume that $(n_2n_3,p)=1$ and that at least one of $m_2n_3-m_3n_2,m_2n_2+m_3n_3\in \Z/p^{\lfloor r/4\rfloor}\Z$ is non-zero.
\item[-] If $p\equiv3\,(\mo 4)$ and $r=2$, assume that $(m_2,m_3,p)=(n_2,n_3,p)=(m_2n_3+m_3n_2,p)=1$. 
\item[-] If $p\equiv3\,(\mo 4)$ and $r=4t+2$ for some $t\in\Z_{\geq1}$, assume that $(m_2^2+m_3^2,p)=1$.
\item[-] If $p=2$ is and $r$ is even, assume that $(m_2m_3,p)=1$.
\end{itemize}
Then $S_4(m_2,m_3,n_2,n_3;p^r) \ll 
p^{r/2}$.
\end{theorem}

The explicit methods in this paper can in principle also be used to compute and estimate Kloosterman sums for other Weyl elements and other groups. The author hopes that this initiates further research in this area.

This paper is organized as follows. In Section \ref{sec:evaluating_sums}, we provide the tools that we will use to obtain explicit expressions for Kloosterman sums: a result of Stevens and Pl\"ucker coordinates. In Sections \ref{sec:so33_sums_section} and \ref{sec:so42_sums_section}, we study Kloostermans sums over $SO_{3,3}$ and $SO_{4,2}$, respectively.

\section*{Acknowledgements}
The author would like to thank Valentin Blomer for providing the topic as well as for the valuable guidance on this project, which started as part of a master's thesis at the University of Bonn. The author also expresses much gratitude to Alberto Acosta Reche for helpful suggestions. This work was supported by the Engineering and Physical Sciences Research Council [EP/S021590/1], via the EPSRC Centre for Doctoral Training in Geometry and Number Theory (The London School of Geometry and Number Theory), University College London.
\section{Evaluating Kloosterman sums}\label{sec:evaluating_sums}
In this section, we state some results which will aid us in obtaining an explicit expression for Kloosterman sums.

The first result allows us to restrict our attention to Kloosterman sums associated to only one representative in $N(\Q_p)$ per element in $W(\Q_p)$:
\begin{theorem}[{\cite[Theorem 3.2]{S}}]\label{prop:stevens} 
Let $G$ be a reductive group over $\Q_p$ with split maximal torus $T$. Let $n\in N(\Q_p)$ and consider complex characters $\psi,\psi':U(\Q_p)\rightarrow\C^\times$ that are trivial on $U(\Z_p)$. For $t\in T(\Z_p)$, define modified characters $\psi_t,\psi'_t:U(\Q_p)\rightarrow\C^\times$ by $\psi_t(u)= \psi(tut^{-1})$, $\psi'_t(u)= \psi'(tut^{-1})$. Then
\begin{align*}
\kl_p(\psi_t,\psi';n) = \kl_p(\psi,\psi';tn)\quad \text{ and } \quad
\kl_p(\psi,\psi'_t;n) = \kl_p(\psi,\psi';nt^{-1}).
\end{align*}
\end{theorem}

\begin{remark}\label{rmk:stevens_non_split}
If $G$ is non-split reductive over $\Q_p$, then the previous proposition statement is still true under the modification that $t\in C(\Z_p)$, where $C= C_G(S)$ is the centralizer of a maximal $\Q_p$-split torus $S$ of $G$.
\end{remark}

A common strategy for studying Kloosterman sums is by Pl\"ucker coordinates, which provide an explicit parametrization of the Bruhat decomposition of $G(\Q_p)$. A certain set of relations and coprimality conditions among these coordinates guarantee the existence of an integral matrix in the corresponding orbit of $U(\Q_p)\backslash G(\Q_p)$. We state the results of Friedberg \cite{F}, which treat the case of $G=SL_n$, and later apply these to our orthogonal groups.

So let us consider $G= SL_n(\Q_p)$, $n\geq2$. Write $T$ for the maximal torus given by the subgroup of diagonal matrices, $U$ for the standard upper-triangular unipotent subgroup, $W$ for the Weyl group of which the elements are monomial matrices with entries in $\{\pm1\}$. For $1\leq k< n$, denote by $L_k$ the collection of subsets of $k$ elements of $\{1,\dotsc,n\}$. Write each such subset as $\lambda_{k,i}=(l_{k,i,1},l_{k,i,2},\dotsc,l_{k,i,k})$ (with $l_{k,i,1}<l_{k,i,2}<\dotsc<\dotsc<l_{k,i,k}$) for $1\leq i\leq {n \choose k}$, such that 
$$L_{k}=(\lambda_{k,1},\dotsc,\lambda_{k,{n\choose k}})$$
in lexicographic order. For a matrix $A\in G$, let $M_\lambda(A)$ be the minor formed by the $\lvert\lambda\rvert$ bottom rows and the columns indexed by $\lambda$ of $A$. Fixing a Weyl element $w\in W$, write $\omega\in S_n$ for the permutation associated to $w$; meaning, $we_{\omega(i)}=e_i$, so $w$ is nonzero precisely in the entries $(i,\omega(i))$ for $1\leq i\leq n$. Here $\{e_i\}_{1\leq i\leq n}$ denotes the standard basis for $\Q_p^n$. Write also $G_w=UTwU=UwTU$ for the corresponding Bruhat cell.

\begin{prop}[{\cite[Proposition 3.1]{F}}]\label{thm:friedberg1}
Assume that some matrix $A\in G_w$ has Bruhat decomposition $A=utwu'$. Let $\lambda_k=\left\{\omega(n),\omega(n-1),\dotsc,\omega(n-k+1)\right\}$, $1\leq k<n$. Then  $t$ is of the form $\diag(1/t_{n-1},t_{n-1}/t_{n-2},\dotsc,t_2/t_1,t_1)$, with $t_k = M_{\lambda_k}(A)$ up to multiplication by $-1$. Moreover,
$$G_w=\left\{A\in G\,:\, M_{\lambda_k}(A)\neq0,M_\lambda(A)=0\text{ for all } \lambda\in L_k,\lambda<\lambda_k,1\leq k<n\right\}.$$
\end{prop}

Since $\sum_{k=1}^{n-1}{n \choose k}=2^n-2$, we can index the coordinates of $\Q_p^{2^n-2}$ by $L_1\cup\dotsc\cup L_{n-1}$. Let 
$$V\subseteq \Q_p^{2^n-2}\simeq\wedge^1(\Q_p^n)\times \dotsb\wedge^{n-1}(\Q_p^n)$$
be the affine algebraic set with coordinates satisfying 
\begin{equation}\label{eq:plucker_image}
\begin{pmatrix}
    v_{\lambda_{k+1,1}} \\
    \vdots \\
    v_{\lambda_{k+1,{n \choose k+1}}}
\end{pmatrix}\in \im T_k
\end{equation}
for $1\leq k < n-1$, where $T_k=(t_{ij})$ denotes the linear transformation given by
$$t_{ij}=\begin{cases}
    (-1)^{r-1}v_{\lambda_{k+1,i}\backslash\{j\}} &\text{ if } j=l_{k+1,i,r}\in\lambda_{k+1,i} \\
    0 &\text{ otherwise}
\end{cases},\quad\quad 1\leq i\leq {n\choose k+1},1\leq j\leq n.$$
Note that this formula reflects the computation of the determinant of a matrix in terms of an iterative expansion of minors, the minors in each iteration being precisely those of one dimension less (than in the previous iteration) and only covering consecutive rows starting from the bottom of the matrix.

Consider the map
$$M:G\rightarrow \Q_p^{2^n-2}:A\mapsto M(A)=(M_\lambda(A)\,:\,\lambda\in L_1\cup\dotsb\cup L_{n-1}),$$
and call $M(A)$ the (generalized) Plücker coordinates of the matrix $A\in G$. This map is neither injective nor surjective, but turns out to be both after quotienting out the domain on the left by $U$ and restricting the codomain:

\begin{theorem}[{\cite[Theorem 3.3]{F}}]\label{thm:friedberg2}
The map $M$ induces a bijection from $U\backslash G$ to 
$$V_1=\left\{v\in V\,:\, v_\lambda\neq0\text{ for some }\lambda\in L_{n-1}\right\}.$$
\end{theorem}

Write $\Gamma= SL_n(\Z_p)$, $\Gamma_\infty= U\cap \Gamma$. The following refinement of the previous theorem states under which conditions an element of $V_1$ represents an integral matrix:

\begin{theorem}[{\cite[Theorem 3.4]{F}}]\label{thm:friedberg_coprimality}
The map $M$ induces a bijection from $\Gamma_\infty\backslash \Gamma$ to 
$$V_2=\left\{v\in V\,:\, v_\lambda\in \Z_p\text{ for all } \lambda\in\cup_{i=1}^{n-1}L_k,\text{ and }\gcd(v_\lambda\,:\,\lambda\in L_k)=1 \text{ for all }1\leq k<n\right\}.$$
\end{theorem}
\section{$SO_{3,3}$ Kloosterman sums}\label{sec:so33_sums_section}
Throughout this section, we consider the reductive group $G=SO_{3,3}(\Q_p)$ with split maximal torus 
$$T=\left\{\begin{pmatrix}
    t_1 &&&&& \\
    & t_2 &&&& \\
    && t_3 &&& \\
    &&& t_1^{-1} && \\
    &&&& t_2^{-1} & \\
    &&&&& t_3^{-1}
\end{pmatrix}\right\}_{t_1,t_2,t_3\in\Q_p^\times}.$$
The character group of $T$ is given by $X^*(T)\cong\Z\alpha_1\oplus\Z\alpha_2\oplus\Z\alpha_3$, where 
$$\alpha_i:T\rightarrow\G_m:\diag(t_1,t_2,t_3,t_1^{-1},t_2^{-1},t_3^{-1})\mapsto t_i,$$
resulting in a root system $\Phi(G,T)=\left\{\pm\alpha_1\pm\alpha_2,\pm\alpha_1\pm\alpha_3,\pm\alpha_2\pm\alpha_3\right\}$. We choose the base $\Delta=\left\{\alpha_1-\alpha_2,\alpha_2+\alpha_3,\alpha_2-\alpha_3\right\}$, which induces the system of positive roots $\Phi^+=\{\alpha_1\pm\alpha_2,\alpha_1\pm\alpha_3,\alpha_2\pm\alpha_3\}$ and the unipotent subgroup
$$U=\left\{\begin{pmatrix}
    1 & x & y & -xz-yt & z & t \\
    & 1 & u & -z-yv-tu+uvx & -uv & v \\
    && 1 & -t+xv & -v & \\
    &&& 1 && \\
    &&& -x & 1 & \\
    &&& xu-y & -u & 1
\end{pmatrix}\right\}_{x,y,z,t,u,v\in\Q_p}.$$
The Weyl group is given by
\begin{align}\label{eq:so33_weyl}
W= \{&\id, \\
&s_\alpha,s_\beta,s_\gamma,\\
&s_\alpha s_\beta,s_\alpha s_\gamma,s_\beta s_\alpha,s_\beta s_\gamma,s_\gamma s_\beta, \\
&s_\alpha s_\beta s_\gamma,s_\alpha s_\gamma s_\beta,s_\beta s_\alpha s_\gamma, s_\gamma s_\beta s_\alpha, \\
&s_\alpha s_\beta s_\gamma s_\alpha,s_\alpha s_\gamma s_\beta s_\alpha,s_\alpha s_\beta s_\gamma s_\beta,s_\beta s_\alpha s_\gamma s_\beta,s_\gamma s_\beta s_\alpha s_\gamma, \\
&s_\alpha s_\beta s_\gamma s_\alpha s_\beta, s_\alpha s_\beta s_\gamma s_\alpha s_\gamma,s_\alpha s_\beta s_\gamma s_\beta s_\alpha,s_\alpha s_\gamma s_\alpha s_\beta s_\gamma,s_\beta s_\gamma s_\beta s_\alpha s_\beta, \\
&s_\alpha s_\beta s_\alpha s_\gamma s_\beta s_\alpha \,\},
\end{align}
where $\alpha=\alpha_2-\alpha_3$, $\beta=\alpha_1-\alpha_2$, $\gamma= \alpha_2+\alpha_3$, and
$$s_\alpha=\begin{psmallmatrix}
    1 &&&&& \\
    && 1 &&& \\
    & 1 &&&& \\
    &&& 1 && \\
    &&&&& 1 \\
    &&&& 1 &
\end{psmallmatrix},\quad s_\beta=\begin{psmallmatrix}
    & 1 &&&& \\
    1 &&&&& \\
    && 1 &&& \\
    &&&& 1 & \\
    &&& 1 && \\
    &&&&& 1
\end{psmallmatrix},\quad s_\gamma=\begin{psmallmatrix}
    1 &&&&& \\
    &&&&& 1 \\
    &&&& 1 & \\
    &&& 1 && \\
    && 1 &&& \\
    & 1 &&&&
\end{psmallmatrix}.$$

In Section \ref{sec:so33_plucker}, we parametrize the representatives of the Kloosterman set $X(n)$ associated to each Weyl element $\overline{n}\in W$ of length at most $2$, using Plücker coordinates. Based on this information, we construct in Section \ref{sec:so33_kloosterman} the Kloosterman sums associated to these $n$. In Section \ref{sec:so33_bounds}, we discuss non-trivial bounds for the resulting sums.

\subsection{Plücker coordinates}\label{sec:so33_plucker}
For a general matrix $A=(a_{ij})_{1\leq i,j\leq 6}\in G$, we define its \textit{Plücker coordinates} as 
$$v_1=a_{41},\quad v_2=a_{42},\quad v_3=a_{43},\quad v_4=a_{44},\quad v_5=a_{45},\quad v_6=a_{46},$$
$$(v_{ij})_{1\leq i<j\leq 6},\quad(v_{ijk})_{1\leq i<j<k\leq 6},\quad(v_{ijkl})_{1\leq i<j<k<l\leq 6},\quad(v_{ijklm})_{1\leq i<j<k<l<m\leq 6}.$$
Here the last three symbols of the form $v_\lambda$ (for an ordered subset $\lambda$ of $\{1,\dotsc,6\}$) denote the minors of $A$ formed by the $|\lambda|$ bottom rows and the columns indexed by $\lambda$. The $v_{ij}$'s denote the minors of $A$ formed by the fourth and fifth row and the columns $(i,j)$, for $1\leq i<j\leq6$. For instance, $v_{12}=a_{41}a_{52}-a_{42}a_{51}$. This choice was made to ensure that all Plücker coordinates are invariant under left multiplication by $U$. Note that there are $62$ Plücker coordinates in total.

Given the large number of Plücker coordinates, we will not write down a complete set of Plücker relations characterizing $U\backslash G$. For our purposes, it suffices to explicitly compute the Plücker coordinates occurring in a given Bruhat cell $G_w$ of $G$. We will rely on the following result:    
\begin{prop}\label{prop:so33_inj}
The Plücker coordinates defined above induce an injection from $U\backslash G$ to $\Q_p^{62}$. 
\end{prop}
\vspace{-8mm}
\begin{proof}
Let $\gamma$ be a standard upper-triangular unipotent matrix in $SL_6(\Q_p)$ and consider the transformation matrix 
$$v=
\begin{psmallmatrix}
    1 &&&&& \\
    & 1 &&&& \\
    && 1 &&& \\
    &&&&& 1 \\
    &&&& 1 & \\
    &&& 1 &&
\end{psmallmatrix}.$$
By explicit computation, one verifies that $v\gamma v^{-1}\in SO_{3,3}(\Q_p)$ if and only if $v\gamma v^{-1}\in U$. Injectivity of $U\backslash G\rightarrow \Q_p^{62}$ is then implied by Theorem \ref{thm:friedberg2}.
\end{proof}

We have a Bruhat decomposition 
$$G=\sqcup_{w\in W}G_w,\quad G_w= UwTU^w,$$
which consists of twenty-four cells $G_w$, indexed by the elements of $W$. An orbit in $U\backslash G$ is represented by a matrix in $wTU^w$ for some $w\in W$. For a matrix in $wTU^w$, with $w$ of length at most 2, we check (by explicit computation) all possible configurations for its corresponding Plücker coordinates $v=(v_1,\dotsc,v_{23456})\in\Q_p^{62}$:

\begin{enumerate}
\item $\begin{aligned}[t]
w=\id:\quad v\in\{(&0,0,0,v_4,0,0; \\
&0,0,0,0,0,0,0,0,0,0,0,0,v_{45},0,0; \\
&0,0,0,0,0,0,0,0,0,0,0,0,0,0,0,0,0,0,0,v_{456}; \\
&0,0,0,0,0,0,0,0,0,0,0,0,0,0,v_{45}; \\
&0,0,0,0,0,v_4)\,: \,v_4,v_{45},v_{456}\in\Q_p^\times\}
\end{aligned}$

\noindent Per set of coordinates, a representative matrix in $wTU^w$ is given by 
$$\begin{psmallmatrix}
    \frac{1}{v_4} &&&&& \\
    & \frac{v_4}{v_{45}} &&&& \\
    && \frac{v_{45}}{v_{456}} &&& \\
    &&& v_4 && \\
    &&&& \frac{v_{45}}{v_4} & \\
    &&&&& \frac{v_{456}}{v_{45}}
\end{psmallmatrix}=\begin{psmallmatrix}
    1 &&&&& \\
    & 1 &&&& \\
    && 1 &&& \\
    &&& 1 && \\
    &&&& 1 & \\
    &&&&& 1
\end{psmallmatrix}\begin{psmallmatrix}
    \frac{1}{v_4} &&&&& \\
    & \frac{v_4}{v_{45}} &&&& \\
    && \frac{v_{45}}{v_{456}} &&& \\
    &&& v_4 && \\
    &&&& \frac{v_{45}}{v_4} & \\
    &&&&& \frac{v_{456}}{v_{45}}
\end{psmallmatrix}.$$

\item $\begin{aligned}[t]
w=s_\alpha:\quad v\in\{(&0,0,0,v_4,0,0; \\
&0,0,0,0,0,0,0,0,0,0,0,0,v_{45},v_{46},0; \\
&0,0,0,0,0,0,0,0,0,0,0,0,0,0,0,0,0,0,0,v_{456}; \\
&0,0,0,0,0,0,0,0,0,0,0,0,0,-v_{46},v_{45}; \\
&0,0,0,0,0,v_4)\,:\,v_4,v_{46},v_{456}\in\Q_p^\times,\,v_{45}\in\Q_p\}
\end{aligned}$

\noindent Per set of coordinates, a representative matrix in $wTU^w$ is given by
$$\begin{psmallmatrix}
    \frac{1}{v_4} &&&&& \\
    && \frac{v_4}{v_{46}} &&& \\
    & -\frac{v_{46}}{v_{456}} & \frac{v_{45}}{v_{456}} &&& \\
    &&& v_4 && \\
    &&&& \frac{v_{45}}{v_4} & \frac{v_{46}}{v_4} \\
    &&&& -\frac{v_{456}}{v_{46}} 
\end{psmallmatrix}=\begin{psmallmatrix}
    1 &&&&& \\
    && 1 &&& \\
    & 1 &&&& \\
    &&& 1 && \\
    &&&&& 1 \\
    &&&& 1 &
\end{psmallmatrix}\begin{psmallmatrix}
    \frac{1}{v_4} &&&&& \\
    & -\frac{v_{46}}{v_{456}} & \frac{v_{45}}{v_{456}} &&& \\
    && \frac{v_4}{v_{46}} &&& \\
    &&& v_4 && \\
    &&&& -\frac{v_{456}}{v_{46}} & \\
    &&&& \frac{v_{45}}{v_4} & \frac{v_{46}}{v_4}
\end{psmallmatrix}.$$

\item $\begin{aligned}[t]
w=s_\beta:\quad v\in\{(&0,0,0,v_4,v_5,0; \\
&0,0,0,0,0,0,0,0,0,0,0,0,v_{45},0,0; \\
&0,0,0,0,0,0,0,0,0,0,0,0,0,0,0,0,0,0,0,v_{456}; \\
&0,0,0,0,0,0,0,0,0,0,0,0,0,0,v_{45}; \\
&0,0,0,0,-v_5,v_4)\,:\,v_5,v_{45},v_{456}\in\Q_p^\times,\,v_4\in\Q_p\}
\end{aligned}$

\noindent Per set of coordinates, a representative matrix in $wTU^w$ is given by
$$\begin{psmallmatrix}
    & \frac{1}{v_5} &&&& \\
    -\frac{v_5}{v_{45}} & \frac{v_4}{v_{45}} &&&& \\
    && \frac{v_{45}}{v_{456}} &&& \\
    &&& v_4 & v_5 && \\
    &&& -\frac{v_{45}}{v_5} && \\
    &&&&& \frac{v_{456}}{v_{45}}
\end{psmallmatrix}=\begin{psmallmatrix}
    & 1 &&&& \\
    1 &&&&& \\
    && 1 &&& \\
    &&&& 1 & \\
    &&& 1 && \\
    &&&&& 1
\end{psmallmatrix}\begin{psmallmatrix}
    -\frac{v_5}{v_{45}} & \frac{v_4}{v_{45}} &&&& \\
    & \frac{1}{v_5} &&&& \\
    && \frac{v_{45}}{v_{456}} &&& \\
    &&& -\frac{v_{45}}{v_5} && \\
    &&& v_4 & v_5 & \\
    &&&&& \frac{v_{456}}{v_{45}}
\end{psmallmatrix}.$$

\item $\begin{aligned}[t]
w=s_\gamma:\quad v\in\{(&0,0,0,v_4,0,0; \\
&0,0,0,0,0,0,0,0,0,v_{34},0,0,v_{45},0,0; \\
&0,0,0,0,0,0,0,0,0,0,v_{234},0,0,\frac{v_{45}v_{234}}{v_{34}},0,0,0,\frac{v_{45}v_{234}}{v_{34}},0,\frac{v_{45}^2v_{234}}{v_{34}^2}; \\
&0,0,0,0,0,0,0,0,0,0,-v_{34},0,0,0,v_{45}; \\
&0,0,0,0,0,v_4)\,:\, v_4,v_{34},v_{234}\in\Q_p^\times,\, v_{45}\in\Q_p\}
\end{aligned}$

\noindent Per set of coordinates, a representative matrix in $wTU^w$ is given by
$$\begin{psmallmatrix}
    \frac{1}{v_4} &&&&& \\
    &&&&& -\frac{v_4}{v_{34}} \\
    &&&& \frac{v_{34}}{v_{234}} & \\
    &&& v_4 && \\
    && -\frac{v_{34}}{v_4} && \frac{v_{45}}{v_4} & \\
    & \frac{v_{234}}{v_{34}} &&&& \frac{v_{45}v_{234}}{v_{34}^2}
\end{psmallmatrix}=\begin{psmallmatrix}
    1 &&&&& \\
    &&&&& 1 \\
    &&&& 1 & \\
    &&& 1 && \\
    && 1 &&& \\
    & 1 &&&&
\end{psmallmatrix}\begin{psmallmatrix}
    \frac{1}{v_4} &&&&& \\
    & \frac{v_{234}}{v_{34}} &&&& \frac{v_{45}v_{234}}{v_{34}^2} \\
    && -\frac{v_{34}}{v_4} && \frac{v_{45}}{v_4} & \\
    &&& v_4 && \\
    &&&& \frac{v_{34}}{v_{234}} & \\
    &&&&& -\frac{v_4}{v_{34}}
\end{psmallmatrix}.$$

\item $\begin{aligned}[t]w=s_\alpha s_\beta:\quad
v\in\{(&0,0,0,v_4,v_5,0; \\
&0,0,0,0,0,0,0,0,0,0,0,0,v_{45},\frac{v_4v_{56}}{v_5},v_{56}; \\
&0,0,0,0,0,0,0,0,0,0,0,0,0,0,0,0,0,0,0,v_{456}; \\
&0,0,0,0,0,0,0,0,0,v_{56},0,0,0,-\frac{v_4v_{56}}{v_5},v_{45}; \\
&0,0,0,0,-v_5,v_4)\,:\,v_5,v_{56},v_{456}\in\Q_p^\times, v_4,v_{45}\in\Q_p\}\end{aligned}$

\noindent Per set of coordinates, a representative matrix in $wTU^w$ is given by
$$\begin{psmallmatrix}
    & \frac{1}{v_5} &&&& \\
    && \frac{v_5}{v_{56}} &&& \\
    \frac{v_{56}}{v_{456}} & -\frac{v_4v_{56}}{v_5v_{456}} & \frac{v_{45}}{v_{456}} &&& \\
    &&& v_4 & v_5 & \\
    &&& -\frac{v_{45}}{v_5} && \frac{v_{56}}{v_5} \\
    &&& \frac{v_{456}}{v_{56}} && 
\end{psmallmatrix}=\begin{psmallmatrix}
    & 1 &&&& \\
    && 1 &&& \\
    1 &&&&& \\
    &&&& 1 & \\
    &&&&& 1 \\
    &&& 1 &&
\end{psmallmatrix}\begin{psmallmatrix}
    \frac{v_{56}}{v_{456}} & -\frac{v_4v_{56}}{v_5v_{456}} & \frac{v_{45}}{v_{456}} &&& \\
    & \frac{1}{v_5} &&&& \\
    && \frac{v_5}{v_{56}} &&& \\
    &&& \frac{v_{456}}{v_{56}} && \\
    &&& v_4 & v_5 & \\
    &&& -\frac{v_{45}}{v_5} && \frac{v_{56}}{v_5}
\end{psmallmatrix}.$$

\item $\begin{aligned}[t]w=s_\alpha s_\gamma: \quad
v\in\{(&0,0,0,v_4,0,0; \\
&0,0,0,0,0,0,v_{24},0,0,v_{34},0,0,-\frac{v_{34}v_{46}}{v_{24}},v_{46},0; \\
&0,0,0,0,0,0,0,0,0,0,v_{234},0,0,-\frac{v_{46}v_{234}}{v_{24}},0,0,0,-\frac{v_{46}v_{234}}{v_{24}},0,\frac{v_{46}^2v_{234}}{v_{24}^2}; \\
&0,0,0,0,0,0,0,0,0,0,-v_{34},v_{24},0,-v_{46},-\frac{v_{34}v_{46}}{v_{24}}; \\
&0,0,0,0,0,v_4)\,:\,v_4,v_{24},v_{234}\in\Q_p^\times, v_{34},v_{46}\in\Q_p\}
\end{aligned}$

\noindent Per set of coordinates, a representative matrix in $wTU^w$ is given by 
$$\begin{psmallmatrix}
    \frac{1}{v_4} &&&&& \\
    &&&& -\frac{v_4}{v_{24}} & \\
    &&&& \frac{v_{34}}{v_{234}} & -\frac{v_{24}}{v_{234}} \\
    &&& v_4 && \\
    & -\frac{v_{24}}{v_4} & -\frac{v_{34}}{v_4} && -\frac{v_{34}v_{46}}{v_4v_{24}} & \frac{v_{46}}{v_4} \\
    && -\frac{v_{234}}{v_{24}} && -\frac{v_{46}v_{234}}{v_{24}^2} &
\end{psmallmatrix}=\begin{psmallmatrix}
    1 &&&&& \\
    &&&& 1 & \\
    &&&&& 1 \\
    &&& 1 && \\
    & 1 &&&& \\
    && 1 &&&
\end{psmallmatrix}\begin{psmallmatrix}
    \frac{1}{v_4} &&&&& \\
    & -\frac{v_{24}}{v_4} & -\frac{v_{34}}{v_4} && -\frac{v_{34}v_{46}}{v_4v_{24}} & \frac{v_{46}}{v_4} \\
    && -\frac{v_{234}}{v_{24}} && -\frac{v_{46}v_{234}}{v_{24}^2} & \\
    &&& v_4 && \\
    &&&& -\frac{v_4}{v_{24}} & \\
    &&&& \frac{v_{34}}{v_{234}} & -\frac{v_{24}}{v_{234}}
\end{psmallmatrix}.$$

\item $\begin{aligned}[t]w=s_\beta s_\alpha:\quad
v\in\{(&0,0,0,v_4,v_5,v_6; \\
&0,0,0,0,0,0,0,0,0,0,0,0,\frac{v_5v_{46}}{v_6},v_{46},0; \\
&0,0,0,0,0,0,0,0,0,0,0,0,0,0,0,0,0,0,0,v_{456}; \\
&0,0,0,0,0,0,0,0,0,0,0,0,0,-v_{46},\frac{v_5v_{46}}{v_6}; \\
&0,0,0,v_6,-v_5,v_4)\,:\,v_6,v_{46},v_{456}\in\Q_p^\times,v_4,v_5\in\Q_p\}\end{aligned}$

\noindent Per set of coordinates, a representative matrix in $wTU^w$ is given by
$$\begin{psmallmatrix}
    && \frac{1}{v_6} &&& \\
    -\frac{v_6}{v_{46}} && \frac{v_4}{v_{46}} &&& \\
    & -\frac{v_{46}}{v_{456}} & \frac{v_5v_{46}}{v_6v_{456}} &&& \\
    &&& v_4 & v_5 & v_6 \\
    &&& -\frac{v_{46}}{v_6} && \\
    &&&& -\frac{v_{456}}{v_{46}} &
\end{psmallmatrix}=\begin{psmallmatrix}
    && 1 &&& \\
    1 &&&&& \\
    & 1 &&&& \\
    &&&&& 1 \\
    &&& 1 && \\
    &&&& 1 &
\end{psmallmatrix}\begin{psmallmatrix}
    -\frac{v_6}{v_{46}} && \frac{v_4}{v_{46}} &&& \\
    & -\frac{v_{46}}{v_{456}} & \frac{v_5v_{46}}{v_6v_{456}} &&& \\
    && \frac{1}{v_6} &&& \\
    &&& -\frac{v_{46}}{v_6} && \\
    &&&& -\frac{v_{456}}{v_{46}} & \\
    &&& v_4 & v_5 & v_6
\end{psmallmatrix}.$$

\item $\begin{aligned}[t]w=s_\beta s_\gamma: \quad
v\in\{(&0,0,v_3,v_4,v_5,0; \\
&0,0,0,0,0,0,0,0,0,v_{34},0,0,-\frac{v_5v_{34}}{v_3},0,0; \\
&0,0,0,0,0,0,0,0,0,0,v_{234},0,0,-\frac{v_5v_{234}}{v_3},0,0,0,-\frac{v_5v_{234}}{v_3},0,\frac{v_5^2v_{234}}{v_3^2}; \\
&0,0,0,0,0,0,0,0,0,0,-v_{34},0,0,0,-\frac{v_5v_{34}}{v_3}; \\
&-v_3,0,0,0,-v_5,v_4)\,:\,v_3,v_{34},v_{234}\in\Q_p^\times,v_4,v_5\in\Q_p\} \end{aligned}$

\noindent Per set of coordinates, a representative matrix in $wTU^w$ is given by
$$\begin{psmallmatrix}
    &&&&& \frac{1}{v_3} \\
    \frac{v_3}{v_{34}} &&&&& -\frac{v_4}{v_{34}} \\
    &&&& \frac{v_{34}}{v_{234}} & \\
    && v_3 & v_4 & v_5 & \\
    &&& \frac{v_{34}}{v_3} && \\
    & \frac{v_{234}}{v_{34}} &&&& -\frac{v_5v_{234}}{v_3v_{34}}
\end{psmallmatrix}=\begin{psmallmatrix}
    &&&&& 1 \\
    1 &&&&& \\
    &&&& 1 & \\
    && 1 &&& \\
    &&& 1 && \\
    & 1 &&&&
\end{psmallmatrix}\begin{psmallmatrix}
    \frac{v_3}{v_{34}} &&&&& -\frac{v_4}{v_{34}} \\
    & \frac{v_{234}}{v_{34}} &&&& -\frac{v_5v_{234}}{v_3v_{34}} \\
    && v_3 & v_4 & v_5 & \\
    &&& \frac{v_{34}}{v_3} && \\
    &&&& \frac{v_{34}}{v_{234}} & \\
    &&&&& \frac{1}{v_3}
\end{psmallmatrix}.$$

\item $\begin{aligned}[t]w=s_\gamma s_\beta: \quad
v\in\{(&0,0,0,v_4,v_5,0; \\
&0,0,0,0,0,0,0,0,0,\frac{v_4v_{35}}{v_5},v_{35},0,v_{45},0,0; \\
&0,0,0,0,\frac{v_4v_{135}}{v_5},v_{135},0,\frac{v_{45}v_{135}}{v_{35}},0,0,-\frac{v_4^2v_{135}}{v_5^2},-\frac{v_4v_{135}}{v_5},0, \\
&-\frac{v_4v_{45}v_{135}}{v_5v_{35}},0,0,0,-\frac{v_4v_{45}v_{135}}{v_5v_{35}},-\frac{v_{45}v_{135}}{v_{35}}, -\frac{v_{45}^2v_{135}}{v_{35}^2}; \\
&0,0,0,0,0,0,v_{35},0,0,0,-\frac{v_4v_{35}}{v_5},0,0,0,v_{45}; \\
&0,0,0,0,-v_5,v_4)\,:\,v_5,v_{35},v_{135}\in\Q_p^\times,v_4,v_{45}\in\Q_p\}\end{aligned}$

\noindent Per set of coordinates, a representative matrix in $wTU^w$ is given by 
$$\begin{psmallmatrix}
    & \frac{1}{v_5} &&&& \\
    &&&&& -\frac{v_5}{v_{35}} \\
    &&& \frac{v_{35}}{v_{135}} && \\
    &&& v_4 & v_5 & \\
    && -\frac{v_{35}}{v_5} & -\frac{v_{45}}{v_5} && \\
    \frac{v_{135}}{v_{35}} & -\frac{v_4v_{135}}{v_5v_{35}} &&&& -\frac{v_{45}v_{135}}{v_{35}^2}
\end{psmallmatrix}=\begin{psmallmatrix}
    & 1 &&&& \\
    &&&&& 1 \\
    &&& 1 && \\
    &&&& 1 & \\
    && 1 &&& \\
    1 &&&&&
\end{psmallmatrix}\begin{psmallmatrix}
    \frac{v_{135}}{v_{35}} & -\frac{v_4v_{135}}{v_5v_{35}} &&&& -\frac{v_{45}v_{135}}{v_{35}^2} \\
    & \frac{1}{v_5} &&&& \\
    && -\frac{v_{35}}{v_5} & -\frac{v_{45}}{v_5} && \\
    &&& \frac{v_{35}}{v_{135}} && \\
    &&& v_4 & v_5 & \\
    &&&&& -\frac{v_5}{v_{35}}
\end{psmallmatrix}.$$
\end{enumerate}

Let $\Gamma\coloneqq SO_{3,3}(\Z_p)$, $\Gamma_\infty\coloneqq U\cap\Gamma$, and $\Gamma_w\coloneqq U^w\cap\Gamma$ for $w\in W$. Recall that, in order to construct a Kloosterman sum associated to some $n_w\in N$, $w\in W$, we consider the double coset space $\Gamma_\infty\backslash(G_w\cap \Gamma)/\Gamma_w$ of the set $G_w\cap\Gamma$ of integral elements of the cell $G_w$. By the following lemma, the double coset space is well-defined:

\begin{lemma}\label{lemma:so33}
$\Gamma_w$ acts freely on the right of $\Gamma_\infty\backslash(G_w\cap\Gamma)$.
\end{lemma}
\vspace{-7mm}
\begin{proof}
Let $\gamma\in\Gamma_w$ be some element that fixes $\Gamma_\infty utwu'$, where $u\in U$, $u'\in U^w$ and $t\in T$ are as in the unique Bruhat decomposition in \cite[21.80]{jsm2}. By \cite[21.78]{jsm2}, the multiplication map $U_w\times U^w\rightarrow U$, where $U_w\coloneqq U\cap w^{-1}Uw$, is an isomorphism. Since $\Gamma_\infty utwu'=\Gamma_\infty utwu'\gamma$, we have $u'\gamma(u')^{-1}\in U_w\cap U^w=\{I_6\}$ and therefore $\gamma=I_6$.
\end{proof}
\vspace{-5mm}
Also note that, by Theorem \ref{prop:stevens}, we can restrict our attention to the double coset spaces 
$$R_w\coloneqq(UD\cap \Gamma)\backslash(G_w\cap\Gamma)/\Gamma_w,\quad D\coloneqq T\cap\Gamma.$$
As a practical consequence, it suffices to consider Kloosterman sums only for elements $n_w\in N$ of which the entries are all powers of $p$.

Note that Proposition \ref{prop:so33_inj} gives no information about the image of $\Gamma_\infty\backslash \Gamma\rightarrow U\backslash G\rightarrow\Q_p^{62}$. By Theorem \ref{thm:friedberg_coprimality}, for an orbit in $U\backslash G$ to contain an element of $\Gamma$, it is necessary that its Plücker coordinates $v=(v_1,\dotsc,v_{23456})\in\Z_p^{62}$ satisfy
\begin{align}\label{eq:so33_coprimality}
&\left(\{v_i\}_{1\leq i\leq 6}\right)=\left(\{v_{ij}\}_{1\leq i<j\leq 6}\right)=\left(\{v_{ijk}\}_{1\leq i<j<k\leq6}\right) \\
&\quad =\left(\{v_{ijkl}\}_{1\leq i<j<k<l\leq6}\right)=\left(\{v_{ijklm}\}_{1\leq i<j<k<l<m\leq6}\right)=1. \notag
\end{align}
Per Weyl element $w\in W$, given a vector $v\in\Z_p^{62}$ satisfying \eqref{eq:so33_coprimality} as well as the explicit description of $v$ for the given $w$ in the enumeration above, we will explicitly construct a matrix in $\Gamma$ that has $v$ as its Plücker coordinates in Section \ref{sec:so33_kloosterman}. For now, we already assume the existence of these integral matrices, so that we can fully parametrize the coset representatives of $R_w$ in terms of Plücker coordinates:
\vspace{-3mm}
\begin{enumerate}
\item $w=\id$: We have $v_4,v_{45},v_{456}\in\Z_p^\times$, so 
\begin{equation*}
\begin{split}
R_{\id}=\{(&0,0,0,1,0,0; \\
&0,0,0,0,0,0,0,0,0,0,0,0,1,0,0; \\
&0,0,0,0,0,0,0,0,0,0,0,0,0,0,0,0,0,0,0,1; \\
&0,0,0,0,0,0,0,0,0,0,0,0,0,0,1; \\
&0,0,0,0,0,1)\}.
\end{split}
\end{equation*}

\item $w= s_\alpha$: We have $v_4,v_{456}\in\Z_p^\times$. Moreover, by the expression of $\Gamma_w$, we have
\begin{equation*}
\begin{split}
R_{s_\alpha}=\{(&0,0,0,1,0,0; \\
&0,0,0,0,0,0,0,0,0,0,0,0,v_{45},p^r,0; \\
&0,0,0,0,0,0,0,0,0,0,0,0,0,0,0,0,0,0,0,1; \\
&0,0,0,0,0,0,0,0,0,0,0,0,0,-p^r,v_{45}; \\
&0,0,0,0,0,1)\},
\end{split}
\end{equation*}
where $r\in\Z_{\geq0}$, and $v_{45}\,(\mo p^r)$ such that, if $r>0$, then $v_{45}\in\Z_p^\times$.

\item $w=s_\beta$: We have $v_{45},v_{456}\in\Z_p^\times$. Moreover, by the expression of $\Gamma_w$, we have
\begin{equation*}
\begin{split}
R_{s_\beta}=\{(&0,0,0,v_4,p^r,0; \\
&0,0,0,0,0,0,0,0,0,0,0,0,1,0,0; \\
&0,0,0,0,0,0,0,0,0,0,0,0,0,0,0,0,0,0,0,1; \\
&0,0,0,0,0,0,0,0,0,0,0,0,0,0,1; \\
&0,0,0,0,-p^r,v_4)\},
\end{split}
\end{equation*}
where $r\in\Z_{\geq0}$, and $v_4\,(\mo p^r)$ such that, if $r>0$, then $v_4\in\Z_p^\times$.

\item $w=s_\gamma$: We have $v_4\in\Z_p^\times$. Moreover, by the expression of $\Gamma_w$, we have 
\begin{equation*}
\begin{split}
R_{s_\gamma}=\{(&0,0,0,1,0,0; \\
&0,0,0,0,0,0,0,0,0,p^r,0,0,v_{45},0,0; \\
&0,0,0,0,0,0,0,0,0,0,p^s,0,0,\frac{v_{45}p^s}{p^r},0,0,0,\frac{v_{45}p^s}{p^r},0,\frac{v_{45}^2p^s}{p^{2r}}; \\
&0,0,0,0,0,0,0,0,0,0,-p^r,0,0,0,v_{45}; \\
&0,0,0,0,0,1)\},
\end{split}
\end{equation*}
where $r,s\in\Z_{\geq0}$, and $v_{45}\,(\mo p^r)$, such that $p^r\vert v_{45}p^s$ as well as $p^{2r}\vert v_{45}^2p^s$ and
$$\left(p^s,\frac{v_{45}p^s}{p^r},\frac{v_{45}^2p^s}{p^{2r}}\right)=1$$
and, if $r>0$, then $v_{45}\in\Z_p^\times$.

\item $w=s_\alpha s_\beta$: We have $v_{456}\in\Z_p^\times$. Moreover, by the expression of $\Gamma_w$, we have 
\begin{equation*}
\begin{split}
R_{s_\alpha s_\beta}=\{(&0,0,0,v_4,p^r,0; \\
&0,0,0,0,0,0,0,0,0,0,0,0,v_{45},\frac{v_4p^s}{p^r},p^s; \\
&0,0,0,0,0,0,0,0,0,0,0,0,0,0,0,0,0,0,0,1; \\
&0,0,0,0,0,0,0,0,0,p^s,0,0,0,-\frac{v_4p^s}{p^r},v_{45}; \\
&0,0,0,0,-p^r,v_4)\},
\end{split}
\end{equation*}
where $r,s\in\Z_{\geq0}$, and $v_4\,(\mo p^r)$, $v_{45}\,(\mo p^s)$, such that $p^r\vert v_4p^s$ and
$$\left(v_{45},\frac{v_4p^s}{p^r},p^s\right)=1$$
and, if $r>0$, then $v_4\in\Z_p^\times$.

\item $w=s_\alpha s_\gamma$: We have $v_4\in\Z_p^\times$. Moreover, by the expression of $\Gamma_w$, we have
\begin{equation*}
\begin{split}
R_{s_\alpha s_\gamma}=\{(&0,0,0,1,0,0; \\
&0,0,0,0,0,0,p^r,0,0,v_{34},0,0,-\frac{v_{34}v_{46}}{p^r},v_{46},0; \\
&0,0,0,0,0,0,0,0,0,0,p^s,0,0,-\frac{v_{46}p^s}{p^r},0,0,0,-\frac{v_{46}p^s}{p^r},0,\frac{v_{46}^2p^s}{p^{2r}}; \\
&0,0,0,0,0,0,0,0,0,0,-v_{34},p^r,0,-v_{46},-\frac{v_{34}v_{46}}{p^r}; \\
&0,0,0,0,0,1)\},
\end{split}
\end{equation*}
where $r,s\in\Z_{\geq0}$, and $v_{34},v_{46}\,(\mo p^r)$, such that $p^r\vert v_{34}v_{46},v_{46}p^s$ as well as $p^{2r}\vert v_{46}^2p^s$ and
$$\left(p^r,v_{34},v_{46},\frac{v_{34}v_{46}}{p^r}\right)=\left(p^s,\frac{v_{46}p^s}{p^r},\frac{v_{46}^2p^s}{p^{2r}}\right)=1.$$

\item $w=s_\beta s_\alpha$: We have $v_{456}\in\Z_p^\times$. Moreover, by the expression of $\Gamma_w$, we have
\begin{equation*}
\begin{split}
R_{s_\beta s_\alpha}=\{(&0,0,0,v_4,v_5,p^r; \\
&0,0,0,0,0,0,0,0,0,0,0,0,\frac{v_5p^s}{p^r},p^s,0; \\
&0,0,0,0,0,0,0,0,0,0,0,0,0,0,0,0,0,0,0,1; \\
&0,0,0,0,0,0,0,0,0,0,0,0,0,-p^s,\frac{v_5p^s}{p^r}; \\
&0,0,0,p^r,-v_5,v_4)\},
\end{split}
\end{equation*}
where $r,s\in\Z_{\geq0}$, and $v_4,v_5\,(\mo p^r)$, such that $p^r\vert v_5p^s$ and
$$\left(v_4,v_5,p^r\right)=\left(\frac{v_5p^s}{p^r},p^s\right)=1.$$

\item $w=s_\beta s_\gamma$: By the expression of $\Gamma_w$, we have
\begin{equation*}
\begin{split}
R_{s_\beta s_\gamma}=\{(&0,0,p^r,v_4,v_5,0; \\
&0,0,0,0,0,0,0,0,0,p^s,0,0,-\frac{v_5p^s}{p^r},0,0; \\
&0,0,0,0,0,0,0,0,0,0,p^t,0,0,-\frac{v_5p^t}{p^r},0,0,0,-\frac{v_5p^t}{p^r},0,\frac{v_5^2p^t}{p^{2r}}; \\
&0,0,0,0,0,0,0,0,0,0,-p^s,0,0,0,-\frac{v_5p^s}{p^r}; \\
&-p^r,0,0,0,-v_5,v_4)\},
\end{split}
\end{equation*}
where $r,s,t\in\Z_{\geq0}$, and $v_4,v_5\,(\mo p^r)$, such that $p^r\vert v_5p^s,v_5p^t$ as well as $p^{2r}\vert v_5^2p^t$ and 
$$\left(p^r,v_4,v_5\right)=\left(p^s,\frac{v_5p^s}{p^r}\right)=\left(p^t,\frac{v_5p^t}{p^r},\frac{v_5^2p^t}{p^{2r}}\right)=1.$$

\item $w=s_\gamma s_\beta$: By the expression of $\Gamma_w$, we have 
\begin{align*}
R_{s_\gamma s_\beta}=\{(&0,0,0,v_4,p^r,0; \\
&0,0,0,0,0,0,0,0,0,\frac{v_4p^s}{p^r},p^s,0,v_{45},0,0; \\
&0,0,0,0,\frac{v_4p^t}{p^r},p^t,0,\frac{v_{45}p^t}{p^s},0,0,-\frac{v_4^2p^t}{p^{2r}},-\frac{v_4p^t}{p^r},0, \\
&-\frac{v_4v_{45}p^t}{p^rp^s},0,0,0,-\frac{v_4v_{45}p^t}{p^rp^s},-\frac{v_{45}p^t}{p^s},-\frac{v_{45}^2p^t}{p^{2s}}; \\
&0,0,0,0,0,0,p^s,0,0,0,-\frac{v_4p^s}{p^r},0,0,0,v_{45}; \\
&0,0,0,0,-p^r,v_4)\},
\end{align*}
where $r,s,t\in\Z_{\geq0}$, and $v_4\,(\mo p^r)$, $v_{45}\,(\mo p^s)$, such that $p^r\vert v_4p^s,v_4p^t$ as well as $p^s\vert v_{45}p^t$ and $p^{2r}\vert v_4^2p^t$ and $p^rp^s\vert v_4v_{45}p^t$ and $p^{2s}\vert v_{45}^2p^t$ and
$$\left(\frac{v_4p^s}{p^r},p^s,v_{45}\right)=\left(\frac{v_4p^t}{p^r},p^t,\frac{v_{45}p^t}{p^s},\frac{v_4^2p^t}{p^{2r}},\frac{v_4v_{45}p^t}{p^rp^s},\frac{v_{45}^2p^t}{p^{2s}}\right)=1$$
and, if $r>0$, then $v_4\in\Z_p^\times$.
\end{enumerate}

\subsection{Explicit description}\label{sec:so33_kloosterman}
Characters of $U$ trivial on $\Gamma$ must be of the form $\psi\coloneqq\psi_{m_1,m_2,m_3}$, $\psi'\coloneqq\psi_{n_1,n_2,n_3}$, with $m_1,m_2,m_3,n_1,n_2,n_3\in\Z_p$, where
$$\psi_{m_1,m_2,m_3}\begin{pmatrix}
    1 & x & * & * & * & * \\
    & 1 & y & * & * & z \\
    && 1 & * & -z &  \\
    &&& 1 && \\
    &&& -x & 1 & \\
    &&& * & -y & 1
\end{pmatrix}\coloneqq e(m_1x+m_2y+m_3z)$$
and similarly for $\psi_{n_1,n_2,n_3}$. Per length of the Weyl element in $W$, we write down the explicit Kloosterman sum: 

\begin{enumerate}
\item The only Weyl element of length $0$ is $w=\id$. We have $n_{\id}=I_6$ (up to units) and the Kloosterman sum
$$\kl_p(\psi,\psi';n_{\id})=1$$
is trivial because $R_w$ has only one element, represented by the identity matrix.

\item Regarding Weyl elements of length $1$, we illustrate the procedure of obtaining a Kloosterman sum for $w=s_\alpha$. From Section \ref{sec:so33_plucker}, we know that a representative in $G_w$ for an arbitrary element $v=(v_1,\dotsc,v_{23456})\in R_w$ is given by
$$A\coloneqq\begin{psmallmatrix}
    1 &&&&& \\
    && \frac{1}{p^r} &&& \\
    & -p^r & v_{45} &&& \\
    &&& 1 && \\
    &&&& v_{45} & p^r \\
    &&&& -\frac{1}{p^r} &
\end{psmallmatrix},$$
where $r\in\Z_{\geq0}$, $v_{45}\,(\mo p^r)$, and, if $r>0$, then $v_{45}\in\Z_p^\times$. Assume for simplicity that $r>0$; otherwise we obtain a trivial Kloosterman sum anyway. The $U$-orbit of $A$ contains an integral matrix: 
\begin{align*}
&\begin{psmallmatrix}
        1 &&&&& \\
        & 1 & -\frac{1}{v_{45}p^r} &&& \\
        && 1 &&& \\
        &&& 1 && \\
        &&&& 1 & \\
        &&&& \frac{1}{v_{45}p^r} & 1
\end{psmallmatrix}A = \begin{psmallmatrix}
        1 &&&&& \\
        & \frac{1}{v_{45}} &&&& \\
        & -p^r & v_{45} &&& \\
        &&& 1 && \\
        &&&& v_{45} & p^r \\
        &&&&& \frac{1}{v_{45}} 
\end{psmallmatrix} \\
&=\begin{psmallmatrix}
        1 &&&&& \\
        & 1 & -\frac{1}{v_{45}p^r} &&& \\
        && 1 &&& \\
        &&& 1 && \\
        &&&& 1 & \\
        &&&& \frac{1}{v_{45}p^r} & 1
\end{psmallmatrix}\begin{psmallmatrix}
    1 &&&&& \\
    && 1 &&& \\
    & 1 &&&& \\
    &&& 1 && \\
    &&&&& 1 \\
    &&&& 1 &
\end{psmallmatrix}\begin{psmallmatrix}
    1 &&&&& \\
    & -p^r &&&& \\
    && \frac{1}{p^r} &&& \\
    &&& 1 && \\
    &&&& -\frac{1}{p^r} & \\
    &&&&& p^r
\end{psmallmatrix}\begin{psmallmatrix}
    1 &&&&& \\
    & 1 & -\frac{v_{45}}{p^r} &&& \\
    && 1 &&& \\
    &&& 1 && \\
    &&&& 1 & \\
    &&&& \frac{v_{45}}{p^r} & 1
\end{psmallmatrix}.
\end{align*}
We conclude that for
$$n_{s_\alpha,r}\coloneqq\begin{psmallmatrix}
    1 &&&&& \\
    && p^{-r} &&& \\
    & -p^r &&&& \\
    &&& 1 && \\
    &&&&& p^r \\
    &&&& -p^{-r} & 
\end{psmallmatrix}$$
for some $r\in\Z_{\geq0}$, we recover the classical Kloosterman sum
\begin{align*}
\kl_p(\psi,\psi';n_{s_\alpha,r}) = \sum_{\substack{v_{45}\,(\mo p^r) \\ (v_{45},p^r)=1}}e\left(-\frac{m_2\overline{v_{45}}+n_2v_{45}}{p^r}\right) = S(m_2,n_2;p^r).
\end{align*}

\noindent In an entirely analogous manner as for $w=s_\alpha$, one recovers the following classical Kloosterman sums for $w\in\{s_\beta,s_\gamma\}$:

\begin{align*}
&\,n_{s_\beta,r} \coloneqq \begin{psmallmatrix}
    & p^{-r} &&&& \\
    -p^r &&&&& \\
    && 1 &&& \\
    &&&& p^r & \\
    &&& -p^{-r} && \\
    &&&&& 1
\end{psmallmatrix},\quad r\in \Z_{\geq0}, \\
&\kl_p(\psi,\psi';n_{s_\beta,r}) = \sum_{\substack{v_4\,(\mo p^r) \\ (v_4,p^r)=1}}e\left(-\frac{m_1\overline{v_4}+n_1v_4}{p^r}\right) = S(m_1,n_1;p^r), \\
&\,n_{s_\gamma,r}\coloneqq\begin{psmallmatrix}
    1 &&&&& \\
    &&&&& -p^{-r} \\
    &&&& p^{-r} & \\
    &&& 1 && \\
    && -p^r &&& \\
    & p^r &&&&
\end{psmallmatrix},\quad r\in \Z_{\geq0}, \\
&\kl_p(\psi,\psi';n_{s_\gamma,r}) = \sum_{\substack{v_{45}\,(\mo p^r) \\ (v_{45},p^r)=1}}e\left(\frac{m_3\overline{v_{45}}+n_3v_{45}}{p^r}\right) = S(m_3,n_3;p^r)
\end{align*}

\noindent The only Weyl element of length 2 that yields similar sums is $w=s_\alpha s_\gamma$, giving a product of two classical Kloosterman sums:
\begin{align*}
&\,n_{s_\alpha s_\gamma,r,s}\coloneqq\begin{psmallmatrix}
    1 &&&&& \\
    &&&& -p^{-r-s} & \\
    &&&&& -p^{s-r} \\
    &&& 1 && \\
    & -p^{r+s} &&&& \\
    && -p^{r-s} &&&
\end{psmallmatrix},\quad r,s\in\Z_{\geq0}, \\
&\kl_p(\psi,\psi';n_{s_\alpha s_\gamma,r,s}) = S(m_2,n_2;p^s)S(m_3,n_3;p^r)
\end{align*}

\item As illustrated by the previous case, the routine of constructing the Kloosterman sum associated to some $w\in W$ consists of the four following steps: (i) write down a representative matrix $A\in G_w$ that has Plücker coordinates $v\in R_w$, (ii) simplify the coprimality conditions on $v$, (iii) provide an element of $\gamma\in U$ such that $\gamma A\in\Gamma$, (iv) deduce the Kloosterman sum for each valid representative $n_w\in N$ of $w$. We continue this routine for the Weyl elements of length $2$, by illustrating the case of $w=s_\alpha s_\beta$.
We have
$$A\coloneqq \begin{psmallmatrix}
    & \frac{1}{p^r} &&&& \\
    && \frac{p^r}{p^s} &&& \\
    p^s & -\frac{v_4p^s}{p^r} & v_{45} &&& \\
    &&& v_4 & p^r & \\
    &&& -\frac{v_{45}}{p^r} && \frac{p^s}{p^r} \\
    &&& \frac{1}{p^s} && 
\end{psmallmatrix},$$
for some $r,s\in\Z_\geq0$, and $v_4\,(\mo p^r)$, $v_{45}\,(\mo p^s)$ such that $\frac{v_4p^s}{p^r}\in\Z_p$ as well as 
\begin{equation}\label{eq:so33_coprimality_alpha_beta}
\left(v_{45},\frac{v_4p^s}{p^r},p^s\right)=1
\end{equation}
and, if $r>0$, then $v_4\in\Z_p^\times$. Assume for simplicity that $r>0$; otherwise we obtain a classical Kloosterman sum anyway. Write $v_{45}=ap^k$ for some $a\in(\Z/p^{s-k}\Z)\cap (\Z_p^\times\cup\{0\})$, and $0\leq k<s$. If $a\neq0$, then the coprimality condition \eqref{eq:so33_coprimality_alpha_beta} translates into $\min\{k,s-r,s\}=0$, so either $k=0$ or $r=s$. If $r<s$, then $k=0$, and we have 
\begin{align*}
&\begin{psmallmatrix}
    1 & -\frac{a}{v_4p^r} & \frac{1}{v_4p^s} &&& \\
    & 1 & -\frac{p^r}{ap^s} &&& \\
    && 1 &&& \\
    &&& 1 && \\
    &&& \frac{a}{v_4p^r} & 1 & \\
    &&&& \frac{p^r}{ap^s} & 1
\end{psmallmatrix}A = \begin{psmallmatrix}
    \frac{1}{v_4} &&&&& \\
    -\frac{p^r}{a} & \frac{v_4}{a} &&&& \\
    p^s & -\frac{v_4p^s}{p^r} & a &&& \\
    &&& v_4 & p^r & \\
    &&&& \frac{a}{v_4} & \frac{p^s}{p^r} \\
    &&&&& \frac{1}{a}
\end{psmallmatrix} \\
&=\begin{psmallmatrix}
    1 & -\frac{a}{v_4p^r} & \frac{1}{v_4p^s} &&& \\
    & 1 & -\frac{p^r}{ap^s} &&& \\
    && 1 &&& \\
    &&& 1 && \\
    &&& \frac{a}{v_4p^r} & 1 & \\
    &&&& \frac{p^r}{ap^s} & 1
\end{psmallmatrix}\begin{psmallmatrix}
    & 1 &&&& \\
    && 1 &&& \\
    1 &&&&& \\
    &&&& 1 & \\
    &&&&& 1 \\
    &&& 1 && 
\end{psmallmatrix}\begin{psmallmatrix}
    p^s &&&&& \\
    & \frac{1}{p^r} &&&& \\
    && \frac{p^r}{p^s} &&& \\
    &&& \frac{1}{p^s} && \\
    &&&& p^r & \\
    &&&&& \frac{p^s}{p^r}
\end{psmallmatrix}\begin{psmallmatrix}
    1 & -\frac{v_4}{p^r} & \frac{a}{p^s} &&& \\
    & 1 &&&& \\
    && 1 &&& \\
    &&& 1 && \\
    &&& \frac{v_4}{p^r} & 1 & \\
    &&& -\frac{a}{p^s} && 1
\end{psmallmatrix}.
\end{align*}

On the other hand, if $r=s$, we have
\begin{align*}
&\begin{psmallmatrix}
    1 & -\frac{ap^k}{v_4p^r} & \frac{1}{v_4p^r} &&& \\
    & 1 &&&& \\
    && 1 &&& \\
    &&& 1 && \\ 
    &&& \frac{ap^k}{v_4p^r} & 1 & \\
    &&& -\frac{1}{v_4p^r} && 1
\end{psmallmatrix}A = \begin{psmallmatrix}
    \frac{1}{v_4} &&&&& \\
    && 1 &&& \\
    p^r & -v_4 & ap^k &&& \\
    &&& v_4 & p^r & \\
    &&&& \frac{ap^k}{v_4} & 1 \\
    &&&& -\frac{1}{v_4} &
\end{psmallmatrix} \\
&= \begin{psmallmatrix}
    1 & -\frac{ap^k}{v_4p^r} & \frac{1}{v_4p^r} &&& \\
    & 1 &&&& \\
    && 1 &&& \\
    &&& 1 && \\ 
    &&& \frac{ap^k}{v_4p^r} & 1 & \\
    &&& -\frac{1}{v_4p^r} && 1
\end{psmallmatrix}\begin{psmallmatrix}
    & 1 &&&& \\
    && 1 &&& \\
    1 &&&&& \\
    &&&& 1 & \\
    &&&&& 1 \\
    &&& 1 && 
\end{psmallmatrix}\begin{psmallmatrix}
    p^r &&&&& \\
    & \frac{1}{p^r} &&&& \\
    && 1 &&& \\
    &&& \frac{1}{p^r} && \\
    &&&& p^r & \\
    &&&&& 1
\end{psmallmatrix}\begin{psmallmatrix}
    1 & -\frac{v_4}{p^r} & \frac{ap^k}{p^r} &&& \\
    & 1 &&&& \\
    && 1 &&& \\
    &&& 1 && \\
    &&& \frac{v_4}{p^r} & 1 & \\
    &&& -\frac{ap^k}{p^r} && 1
\end{psmallmatrix}.
\end{align*}
If $a=0$, then the coprimality condition \eqref{eq:so33_coprimality_alpha_beta} reduces to $r=s$, and the same matrices as for the previous case $\{a\neq0,r=s\}$ can be considered. We conclude that for 
$$n_{s_\alpha s_\beta,r,s}\coloneqq\begin{psmallmatrix}
    & p^{-r} &&&& \\
    && p^{r-s} &&& \\
    p^s &&&&& \\
    &&&& p^r & \\
    &&&&& p^{s-r} \\
    &&& p^{-s}
\end{psmallmatrix}$$
for some $r,s\in\Z_{\geq0}$, $r\leq s$, we recover the $GL_3$ Kloosterman sum (defined in \eqref{eq:gl3_sum})
\begin{align*}
\kl_p(\psi,\psi';n_{s_\alpha s_\beta,r,s}) &= \sum_{\substack{v_4\,(\mo p^r) \\ (v_4,p^r)=1}}\sum_{\substack{a\,(\mo p^s) \\ (a,p^{s-r})=1}}e\left(-\frac{m_1a\overline{v_4}+n_1v_4}{p^r}\right)e\left(-\frac{m_2\overline{a}}{p^{s-r}}\right) \\
&= S_3(-n_1,m_1,m_2;p^r,p^s).
\end{align*}

\noindent In an entirely analogous manner as for $w=s_\alpha s_\beta$, one recovers the following $GL_3$ sums for $w\in\{s_\beta s_\alpha, s_\beta s_\gamma, s_\gamma s_\beta\}$:

\begin{align*}
&\,n_{s_\beta s_\alpha,r,s}\coloneqq\begin{psmallmatrix}
    && p^{-r} &&& \\
    -p^{r-s} &&&&& \\
    & -p^s &&&& \\
    &&&&& p^r \\
    &&& -p^{s-r} && \\
    &&&& -p^{-s} &
\end{psmallmatrix},\quad r,s\in\Z_{\geq0},\,r\geq s, \\
&\kl_p(\psi,\psi';n_{s_\beta s_\alpha,r,s}) = S_3(-n_2,m_2,m_1;p^s,p^r), \\
&\,n_{s_\beta s_\gamma, r,s}\coloneqq\begin{psmallmatrix}
    &&&&& p^{-r} \\
    p^{r-s} &&&&& \\
    &&&& p^{-s} & \\
    && p^r &&& \\
    &&& p^{s-r} && \\
    & p^s &&&&
\end{psmallmatrix},\quad r,s\in\Z_{\geq0},\,r\geq s, \\
&\kl_p(\psi,\psi';n_{s_\beta s_\gamma, r,s}) = S_3(n_3,m_3,m_1;p^s,p^r), \\
&\,n_{s_\gamma s_\beta,r,s}\coloneqq\begin{psmallmatrix}
    & p^{-r} &&&& \\
    &&&&& -p^{r-s} \\
    &&& p^{-s} && \\
    &&&& p^r & \\
    && -p^{s-r} &&& \\
    p^s &&&&&
\end{psmallmatrix},\quad r,s\in\Z_{\geq0},\,r\leq s, \\
&\kl_p(\psi,\psi';n_{s_\gamma s_\beta,r,s}) = S_3(-n_1,m_1,m_3;p^r,p^s)
\end{align*}
\end{enumerate}

\subsection{Bounds}\label{sec:so33_bounds}
The $SO_{3,3}$ Kloosterman sums obtained above can be non-trivially bounded using Weil's bound \eqref{eq:weil_bound} for classical Kloosterman sums and the bound \eqref{eq:larsen_bound} for the $GL_3$ sums.

The strategy of the proof of Theorem \ref{thm:bfg} is to work with case distinction on the values of the parameters $r$ and $s$, and in each case to reduce to a simpler sum that we know how to bound. The case where $p\nmid m_1n_1n_2$ and $s=2r$ is non-trivial: if $r=1$ one applies Deligne's bound for hyper-Kloosterman sums \cite{D}, while if $r>1$ one applies the $p$-adic stationary phase method from \cite{DF}.

\section{$SO_{4,2}$ Kloosterman sums}\label{sec:so42_sums_section}
Throughout this section, consider the reductive group $G=SO_{4,2}(\Q_p)$, $p\not\equiv1\,(\mo 4)$, with maximal $\Q_p$-split torus
$$S=\left\{\begin{psmallmatrix}
    s_1 &&&&& \\ & s_2 &&&& \\ && 1 &&& \\ &&& 1 && \\ &&&& s_1^{-1} & \\ &&&&& s_2^{-1}
\end{psmallmatrix}\right\}_{s_1,s_2\in \Q_p^\times}$$
and centralizer
$$C= C_G(S) = \left\{\begin{psmallmatrix}
    s_1 &&&&& \\
    & s_2 &&&& \\
    && a & b && \\
    && -b & a && \\
    &&&& s_1^{-1} & \\
    &&&&& s_2^{-1}
\end{psmallmatrix}\right\}_{\substack{s_1,s_2\in \Q_p^\times \\ a,b\in \Q_p \\ a^2+b^2=1}}.$$
The character group of $S$ is given by $X^*(S)\cong \Z\alpha_1\oplus\Z\alpha_2$, where 
$$\alpha_i:S\rightarrow\G_m:\diag(s_1,s_2,1,1,s_1^{-1},s_2^{-1})\mapsto s_i,$$
resulting in a relative root system ${}_{\Q_p}\Phi(G,S)=\{\pm\alpha_1,\pm\alpha_2,\pm\alpha_1\pm\alpha_2\}$. We choose the base ${}_{\Q_p}\Delta=\{\alpha_2,\alpha_1-\alpha_2\}$, which induces the system of positive roots ${}_{\Q_p}\Phi^+=\{\alpha_1,\alpha_2,\alpha_1\pm\alpha_2\}$ and the unipotent subgroup
$$U=\left\{\begin{psmallmatrix}
    1 & x & y & z & -\frac{y^2+z^2}{2}-xt & t \\
    & 1 & u & v & \frac{x(u^2+v^2)}{2}-yu-zv-t & -\frac{u^2+v^2}{2} \\
    && 1 && xu-y & -u \\
    &&& 1 & xv-z & -v \\
    &&&& 1 & \\
    &&&& -x & 1
\end{psmallmatrix}\right\}_{x,y,z,t,u,v\in \Q_p}.$$
The relative Weyl group is given by
\begin{align*}
{}_{\Q_p}W = \left\{\id,s_\alpha,s_\beta,s_\alpha s_\beta,s_\beta s_\alpha,s_\alpha s_\beta s_\alpha,s_\beta s_\alpha s_\beta,s_\alpha s_\beta s_\alpha s_\beta\right\},
\end{align*}
where $\alpha= \alpha_2$, $\beta=\alpha_1-\alpha_2$, and
$$s_{\alpha}=\begin{psmallmatrix}
    1 &&&&& \\
    &&&&& 1 \\
    && 1 &&& \\
    &&& -1 && \\
    &&&& 1 & \\
    & 1 &&&&
\end{psmallmatrix},\quad s_\beta=\begin{psmallmatrix}
    & 1 &&&& \\
    1 &&&&& \\
    && 1 &&& \\
    &&& 1 && \\
    &&&&& 1 \\
    &&&& 1 &
\end{psmallmatrix}.$$

In Section \ref{sec:so42_plucker}, we parametrize the representatives of the Kloosterman set $X(n)$ associated to the Weyl elements $\overline{n}\in W$ of length at most $1$, using Plücker coordinates. Based on this information, we construct in Section \ref{sec:so42_explicit} the Kloosterman sum associated to these $n$. In Section \ref{sec:so42_bounds}, we discuss non-trivial bounds for the resulting sums.

\subsection{Pl\"ucker coordinates}\label{sec:so42_plucker}
For a general matrix $A=(a_{ij})_{1\leq i,j\leq 6}\in G$
we define its \textit{Plücker coordinates} as 
\begin{align*}
&\quad\quad\quad v_1=a_{51},\quad v_2=a_{52},\quad v_3=a_{53},\quad v_4=a_{54},\quad v_5=a_{55},\quad v_6=a_{56},\\
&(v_{ij})_{1\leq i<j\leq 6},\quad(v_{ijk})_{1\leq i<j<k\leq 6},\quad(v_{ijkl})_{1\leq i<j<k<l\leq 6},\quad(v_{ijklm})_{1\leq i<j<k<l<m\leq 6}.
\end{align*}
Here, the last four symbols of the form $v_\lambda$ (for an ordered subset $\lambda$ of $\{1,\dotsc,6\}$) denote the minors of $A$ formed by the $\lvert\lambda\rvert$ bottom rows and the columns indexed by $\lambda$. Note that the $v_i$'s, $1\leq i\leq6$, do not denote the $1\times 1$ minors on the bottom row of $A$, which is a technical choice made to ensure that all Plücker coordinates are invariant under left multiplication by $U$. Also note that there are $62$ Plücker coordinates in total.

Similarly as for $SO_{3,3}$, we now explicitly compute the Pl\"ucker coordinates occurring in a given Bruhat cell $G_w$ of $G$. We will rely on the following result, the proof of which is similar to that of Proposition \ref{prop:so33_inj}:

\begin{prop}\label{prop:so42_inj}
The Plücker coordinates defined above induce an injection from $U\backslash G$ to $\Q_p^{62}$.
\end{prop}

We have a relative Bruhat decomposition
$$G = \sqcup_{w\in {}_{\Q_p}W}G_w,\quad G_w\coloneqq UwCU^w,$$
which consists of eight cells $G_w$, indexed by the elements of ${}_{\Q_p}W$. An orbit in $U\backslash G$ is represented by a matrix in $wCU^w$ for some $w\in {}_{\Q_p}W$. For a matrix in $wCU^w$, with $w$ of length at most 1, we check (by explicit computation) all possible configurations for its corresponding Plücker coordinates $v=(v_1,\dotsc,v_{23456})\in\Q_p^{62}$:

\begin{enumerate}
\item $\begin{aligned}[t]w=\id:\quad
v\in\{(&0,0,0,0,v_5,0; \\
&0,0,0,0,0,0,0,0,0,0,0,0,0,0,v_{56}; \\
&0,0,0,0,0,0,0,0,0,0,0,0,0,0,0,0,0,0,v_{356},v_{456}; \\
&0,0,0,0,0,0,0,0,0,0,0,0,0,0,v_{56}; \\
&0,0,0,0,0,{v_5})\,:\,v_5,v_{56}\in\Q_p^\times,v_{356},v_{456}\in\Q_p,v_{356}^2+v_{456}^2=v_{56}^2\}
\end{aligned}$

\noindent Per set of coordinates, a representative matrix in $wCU^w$ is given by 
$$\begin{psmallmatrix}
    \frac{1}{v_5} &&&&& \\
    & \frac{v_5}{v_{56}} &&&& \\
    && \frac{v_{456}}{v_{56}} & -\frac{v_{356}}{v_{56}} && \\
    && \frac{v_{356}}{v_{56}} & \frac{v_{456}}{v_{56}} && \\
    &&&& v_5 & \\
    &&&&& \frac{v_{56}}{v_5}
\end{psmallmatrix}=
\begin{psmallmatrix}
    1 &&&&& \\ & 1 &&&& \\ && 1 &&& \\ &&& 1 && \\ &&&& 1 & \\ &&&&& 1
\end{psmallmatrix}
\begin{psmallmatrix}
    \frac{1}{v_5} &&&&& \\
    & \frac{v_5}{v_{56}} &&&& \\
    && \frac{v_{456}}{v_{56}} & -\frac{v_{356}}{v_{56}} && \\
    && \frac{v_{356}}{v_{56}} & \frac{v_{456}}{v_{56}} && \\
    &&&& v_5 & \\
    &&&&& \frac{v_{56}}{v_5}
\end{psmallmatrix}.$$

\item $\begin{aligned}[t]w=s_\alpha: \quad
v \in\{(&0,0,0,0,v_5,0; \\
&0,0,0,0,0,0,0,v_{25},0,0,v_{35},0,v_{45},0,\frac{v_{35}^2+v_{45}^2}{2v_{25}}; \\
&0,0,0,0,0,0,0,0,0,0,0,v_{235},0,v_{245},0,\frac{v_{35}v_{235}+v_{45}v_{245}}{v_{25}},\\
&\frac{v_{35}v_{245}-v_{45}v_{235}}{v_{25}},0,\frac{v_{235}(v_{35}^2-v_{45}^2)+2v_{35}v_{45}v_{245}}{2v_{25}^2},\frac{v_{245}(v_{45}^2-v_{35}^2)+2v_{35}v_{45}v_{235}}{2v_{25}^2}; \\
&0,0,0,0,0,0,0,0,0,0,-v_{25},0,-v_{45},v_{35},\frac{v_{35}^2+v_{45}^2}{2v_{25}}; \\
&0,0,0,0,0,v_5)\,:\,v_5,v_{25}\in\Q_p^\times,v_{35},v_{45},v_{235},v_{245}\in\Q_p,v_{235}^2+v_{245}^2=v_{25}^2\}\end{aligned}$

\noindent Per set of coordinates, a representative matrix in $wCU^w$ is given by
$$\begin{psmallmatrix}
    \frac{1}{v_5} &&&&& \\ 
    &&&&& -\frac{v_5}{v_{25}} \\
    && \frac{v_{245}}{v_{25}} & -\frac{v_{235}}{v_{25}} && \frac{v_{45}v_{235}-v_{35}v_{245}}{v_{25}^2} \\
    && -\frac{v_{235}}{v_{25}} & -\frac{v_{245}}{v_{25}} && \frac{v_{35}v_{235}+v_{45}v_{245}}{v_{25}^2} \\
    &&&& v_5 & \\
    & -\frac{v_{25}}{v_5} & -\frac{v_{35}}{v_5} & -\frac{v_{45}}{v_5} && \frac{v_{35}^2+v_{45}^2}{2v_5v_{25}}
\end{psmallmatrix}=\begin{psmallmatrix}
    1 &&&&& \\
    &&&&& 1 \\
    && 1 &&& \\
    &&& -1 && \\
    &&&& 1 & \\
    & 1 &&&&
\end{psmallmatrix}\begin{psmallmatrix}
    \frac{1}{v_5} &&&&& \\
    & -\frac{v_{25}}{v_5} & -\frac{v_{35}}{v_5} & -\frac{v_{45}}{v_5} && \frac{v_{35}^2+v_{45}^2}{2v_5v_{25}} \\
    && \frac{v_{245}}{v_{25}} & -\frac{v_{235}}{v_{25}} && \frac{v_{45}v_{235}-v_{35}v_{245}}{v_{25}^2} \\
    && \frac{v_{235}}{v_{25}} & \frac{v_{245}}{v_{25}} && -\frac{v_{35}v_{235}+v_{45}v_{245}}{v_{25}^2} \\
    &&&& v_5 & \\
    &&&&& -\frac{v_5}{v_{25}}
\end{psmallmatrix}.$$

\item $\begin{aligned}[t]w=s_\beta:\quad
v\in\{(&0,0,0,0,v_5,v_6; \\
&0,0,0,0,0,0,0,0,0,0,0,0,0,0,v_{56}; \\
&0,0,0,0,0,0,0,0,0,0,0,0,0,0,0,0,0,0,v_{356},v_{456}; \\
&0,0,0,0,0,0,0,0,0,0,0,0,0,0,v_{56}; \\
&0,0,0,0,-v_6,v_5)\,:\,v_6,v_{56}\in\Q_p^\times,v_5,v_{356},v_{456}\in\Q_p,v_{356}^2+v_{456}^2=v_{56}^2\}\end{aligned}$

\noindent Per set of coordinates, a representative matrix in $wCU^w$ is given by 
$$\begin{psmallmatrix}
    & \frac{1}{v_6} &&&& \\
    -\frac{v_6}{v_{56}} & \frac{v_5}{v_{56}} &&&& \\
    && \frac{v_{456}}{v_{56}} & -\frac{v_{356}}{v_{56}} && \\
    && \frac{v_{356}}{v_{56}} & \frac{v_{456}}{v_{56}} && \\
    &&&& v_5 & v_6 \\
    &&&& -\frac{v_{56}}{v_6} &
\end{psmallmatrix}=\begin{psmallmatrix}
    & 1 &&&& \\
    1 &&&&& \\
    && 1 &&& \\
    &&& 1 && \\
    &&&&& 1 \\
    &&&& 1 &
\end{psmallmatrix}\begin{psmallmatrix}
    -\frac{v_6}{v_{56}} & \frac{v_5}{v_{56}} &&&& \\
    & \frac{1}{v_6} &&&& \\
    && \frac{v_{456}}{v_{56}} & -\frac{v_{356}}{v_{56}} && \\
    && \frac{v_{356}}{v_{56}} & \frac{v_{456}}{v_{56}} && \\
    &&&& -\frac{v_{56}}{v_6} & \\
    &&&& v_5 & v_6
\end{psmallmatrix}.$$
\end{enumerate}

Let $\Gamma\coloneqq SO_{4,2}(\Z_p)$, $\Gamma_\infty\coloneqq U\cap\Gamma$, and $\Gamma_w\coloneqq U^w\cap \Gamma$ for $w\in W$. For a given $w\in {}_{\Q_p}W$, the double coset space $\Gamma_\infty\backslash(G_w\cap\Gamma)/\Gamma_w$ is well-defined (the proof of Lemma \ref{lemma:so33} stays the same, with the only modification being that $t\in C$). As before, it suffices to restrict our attention to the double coset spaces
$$R_w\coloneqq(UD\cap\Gamma)\backslash(G_w\cap\Gamma)/\Gamma_w ,\quad D\coloneqq C\cap\Gamma,$$
by Theorem \ref{prop:stevens} and Remark \ref{rmk:stevens_non_split}.

Note that Proposition \ref{prop:so42_inj} gives no information about the image of $\Gamma_\infty\backslash\Gamma\rightarrow U\backslash G\rightarrow\Q_p^{62}$. By Theorem \ref{thm:friedberg_coprimality}, for an orbit in $U\backslash G$ to contain an element of $\Gamma$, it is necessary that its Plücker coordinates $v=(v_1,\dotsc,v_{23456})\in\Z_p^{62}$ satisfy
\begin{align}\label{eq:so42_coprimality}
&\left(\{v_i\}_{1\leq i\leq 6}\right)=\left(\{v_{ij}\}_{1\leq i<j\leq 6}\right)=\left(\{v_{ijk}\}_{1\leq i<j<k\leq6}\right) \\
&\quad =\left(\{v_{ijkl}\}_{1\leq i<j<k<l\leq6}\right)=\left(\{v_{ijklm}\}_{1\leq i<j<k<l<m\leq6}\right)=1. \notag
\end{align}

\noindent Per relative Weyl element $w\in {}_{\Q_p}W$, given a vector $v\in\Z_p^{62}$ satisfying \eqref{eq:so42_coprimality} as well as the explicit description of $v$ for the given $w$ in the enumeration above, we will explicitly construct a matrix in $\Gamma$ that has $v$ as its Plücker coordinates in Section \ref{sec:so42_explicit}. For now, we already assume the existence of these integral matrices, so that we can fully parametrize the coset representatives of $R_w$ in terms of Plücker coordinates:

\begin{enumerate}
\item $w=\id$: We have $v_5,v_{56}\in\Z_p^\times$ and thus
$\begin{psmallmatrix}
    1 &&&&& \\
    & 1 &&&& \\
    && \frac{v_{456}}{v_{56}} & \frac{v_{356}}{v_{56}} && \\
    && -\frac{v_{356}}{v_{56}} & \frac{v_{456}}{v_{56}} && \\
    &&&& 1 & \\
    &&&&& 1
\end{psmallmatrix}\in D$, so
\begin{align*}
R_{\id}=\{(&0,0,0,0,1,0; \\
&0,0,0,0,0,0,0,0,0,0,0,0,0,0,1; \\
&0,0,0,0,0,0,0,0,0,0,0,0,0,0,0,0,0,0,0,1; \\
&0,0,0,0,0,0,0,0,0,0,0,0,0,0,1; \\
&0,0,0,0,0,1)\}.
\end{align*}

\item $w=s_\alpha$: We have $v_5\in\Z_p^\times$. To see that $\begin{psmallmatrix}
    1 &&&&& \\
    & 1 &&&& \\
    && \frac{v_{245}}{v_{25}} & -\frac{v_{235}}{v_{25}} && \\
    && \frac{v_{235}}{v_{25}} & \frac{v_{245}}{v_{25}} && \\
    &&&& 1 & \\
    &&&&& 1
\end{psmallmatrix}\in D$, keeping in mind that
\begin{equation}\label{eq:so42_orthogonality_alpha}
v_{235}^2+v_{245}^2=v_{25}^2
\end{equation} 
with $v_{25}\in\Z_p\backslash\{0\}$, we consider the following two cases:
\begin{itemize}
\item[-] If $\nu_p(v_{235})\neq\nu_p(v_{245})$, then the fact that $\frac{v_{235}}{v_{25}},\frac{v_{245}}{v_{25}}\in\Z_p$ is a consequence of \eqref{eq:so42_orthogonality_alpha}
by comparing $p$-adic valuations of both sides of the equation. 
\item[-] If $\nu_p(v_{235})=\nu_p(v_{245})$, write $v_{235}=cp^m$, $v_{245}=dp^m$ for some $c,d\in \Z_p^\times$, $m\coloneqq\nu_p(v_{235})\in\Z_{\geq0}$. If $m<\nu_p(v_{25})$, then \eqref{eq:so42_orthogonality_alpha} would imply $c^2+d^2\equiv0\,(\mo p^2)$, which is impossible since $p\not\equiv1\,(\mo 4)$.
\end{itemize}
Moreover, by the expression of $\Gamma_w$, we have
\begin{align*}
R_{s_\alpha} =\{(&0,0,0,0,1,0; \\
&0,0,0,0,0,0,0,p^r,0,0,v_{35},0,v_{45},0,\frac{v_{35}^2+v_{45}^2}{2p^r}; \\
&0,0,0,0,0,0,0,0,0,0,0,0,0,p^r,0,v_{45},v_{35},0,\frac{v_{35}v_{45}}{p^r},-\frac{v_{35}^2-v_{45}^2}{2p^r}; \\
&0,0,0,0,0,0,0,0,0,0,-p^r,0,-v_{45},v_{35},\frac{v_{35}^2+v_{45}^2}{2p^r}; \\
&0,0,0,0,0,1)\},
\end{align*}
for some $r\in\Z_{\geq0}$, $v_{35},v_{45}\,(\mo p^r)$, such that $2p^r\vert v_{35}^2+v_{45}^2,v_{35}^2-v_{45}^2$ as well as $p^r\vert v_{35}v_{45}$ and 
\begin{align*}
\left(p^r,v_{35},v_{45},\frac{v_{35}^2+v_{45}^2}{2p^r}\right)=\bigg(p^r,v_{35},v_{45},\frac{v_{35}v_{45}}{p^r},\frac{v_{35}^2-v_{45}^2}{2p^r}\bigg)=1.
\end{align*}

\item $w=s_\beta$: We have $v_{56}\in\Z_p^\times$ and thus $\begin{psmallmatrix}
    1 &&&&& \\
    & 1 &&&& \\
    && \frac{v_{456}}{v_{56}} & \frac{v_{356}}{v_{56}} && \\
    && -\frac{v_{356}}{v_{56}} & \frac{v_{456}}{v_{56}} && \\
    &&&& 1 & \\
    &&&&& 1
\end{psmallmatrix}\in D$. Moreover, by the expression of $\Gamma_w$, we have
\begin{equation*}
\begin{split}
R_{s_\beta} = \{(&0,0,0,0,v_5,p^r; \\
&0,0,0,0,0,0,0,0,0,0,0,0,0,0,1; \\         
&0,0,0,0,0,0,0,0,0,0,0,0,0,0,0,0,0,0,0,1; \\
&0,0,0,0,0,0,0,0,0,0,0,0,0,0,1; \\
&0,0,0,0,-p^r,v_5)
\end{split}
\end{equation*}
for some $r\in\Z_{\geq0}$, $v_5\,(\mo p^r)$, such that, if $r>0$, then $v_5\in\Z_p^\times$.
\end{enumerate}

\subsection{Explicit description}\label{sec:so42_explicit}
Characters of $U$ trivial on $\Gamma$ must be of the form $\psi\coloneqq\psi_{m_1,m_2,m_3}$, $\psi'\coloneqq\psi_{n_1,n_2,n_3}$, with $m_1,m_2,m_3,n_1,n_2,n_3\in\Z_p$, where
$$\psi_{m_1,m_2,m_3}\begin{pmatrix}
    1 & x & * & * & * & * \\
    & 1 & y & z & * & * \\
    && 1 && * & -y \\
    &&& 1 & * & -z \\
    &&&& 1 & \\
    &&&& -x & 1
\end{pmatrix}\coloneqq e(m_1x+m_2y+m_3z)$$
and similarly for $\psi_{n_1,n_2,n_3}$. Per relative Weyl element $w\in{}_{\Q_p}W$, we write down the explicit Kloosterman sum:
\begin{enumerate}
\item The only Weyl element of length $0$ is $w=\id$. We have $n_{\id}=I_6$ (up to units) and the Kloosterman sum
$$\kl_p(\psi,\psi';n_{\id})=1$$
is trivial because $R_w$ has only one element, represented by the identity matrix.

\item Regarding the Weyl elements of length $1$, we obtain a classical Kloosterman sum for $w=s_\beta$: 
\begin{align*}
&n_{s_\beta,r}\coloneqq\begin{psmallmatrix}
    & p^{-r} &&&& \\
    -p^r &&&&& \\
    && 1 &&& \\
    &&& 1 && \\
    &&&&& p^r \\
    &&&& -p^{-r} &
\end{psmallmatrix},\quad r\in\Z_{\geq0}, \\
&\kl_p(\psi,\psi';n_{s_\beta,r}) = S(m_1,n_1;p^r)
\end{align*}

\noindent This was to be expected, since the long Weyl element 
$$w=\begin{psmallmatrix}
    & 1 && \\
    1 &&& \\
    &&& 1 \\
    && 1 &
\end{psmallmatrix}$$
of $SO_{2,2}(\Q_p)$ also results in a classical Kloosterman sum.\newline

\noindent The Kloosterman sum attached to $w=s_\alpha$ is more intricate. From Section \ref{sec:so42_plucker}, we know that a representative in $G_w$ for an arbitrary element $v=(v_1,\dotsc,v_{23456})\in R_w$ is given by
$$A\coloneqq\begin{psmallmatrix}
    1 &&&&& \\
    &&&&& -\frac{1}{p^r} \\
    && 1 &&& -\frac{v_{35}}{p^r} \\
    &&& -1 && \frac{v_{45}}{p^r} \\
    &&&& 1 & \\
    & -p^r & -v_{35} & -v_{45} && \frac{v_{35}^2+v_{45}^2}{2p^r}
\end{psmallmatrix},$$
where $r\in\Z_{\geq0}$, $v_{35},v_{45}\,(\mo p^r)$, such that $\frac{v_{35}v_{45}}{p^r},\frac{v_{35}^2+v_{45}^2}{2p^r},\frac{v_{35}^2-v_{45}^2}{2p^r}\in\Z_p$ as well as
\begin{align}\label{eq:so42_coprimality_alpha}
\left(p^r,v_{35},v_{45},\frac{v_{35}^2+v_{45}^2}{2p^r}\right)=\bigg(p^r,v_{35},v_{45},\frac{v_{35}v_{45}}{p^r},\frac{v_{35}^2-v_{45}^2}{2p^r}\bigg)=1.
\end{align}

\noindent Assume for simplicity that $r>0$; otherwise we obtain a trivial Kloosterman sum anyway. Write $v_{35}=ap^k$, $v_{45}=bp^l$ for some $a\,(\mo p^{r-k})$, $b\,(\mo p^{r-l})$, $a,b\in\Z_p^\times\cup\{0\}$, and $0\leq l,k<r$.\newline

\noindent For $p\equiv3\,(\mo 4)$, if $a\neq0\neq b$, then the coprimality conditions \eqref{eq:so42_coprimality_alpha} reduce to $r=2\min\{k,l\}$, if $a=0\neq b$ to $r=2l$, and if $a\neq 0=b$ to $r=2k$. For $p=2$, if $a\neq0\neq b$, then the coprimality conditions \eqref{eq:so42_coprimality_alpha} reduce to $r=2\min\{k,l\}-1$ if $k\neq l$ and to $r=2k$ if $k=l$,\footnote{The sum of two squares in $\Z_2^\times$ is always $\equiv 2\,(\mo 8)$.} if $a=0\neq b$ to $r=2l-1$, and if $a\neq0=b$ to $r=2k-1$. Note that thus, for any $p\not\equiv1\,(\mo 4)$, the condition $\frac{a^2p^{2k}+b^2p^{2l}}{2p^{r}}\in\Z_p^\times$ is automatically satisfied.\newline

\noindent For every prime $p\not\equiv1\,(\mo 4)$, we now have
\begin{align*}
&\begin{psmallmatrix}
    1 &&&&& \\
    & 1 & -\frac{2ap^k}{a^2p^{2k}+b^2p^{2l}} & \frac{2bp^l}{a^2p^{2k}+b^2p^{2l}} && -\frac{2}{a^2p^{2k}+b^2p^{2l}} \\
    && 1 &&& \frac{2ap^k}{a^2p^{2k}+b^2p^{2l}} \\
    &&& 1 && -\frac{2bp^l}{a^2p^{2k}+b^2p^{2l}} \\
    &&&& 1 & \\
    &&&&& 1
\end{psmallmatrix}A = \begin{psmallmatrix}
    1 &&&&& \\
    & \frac{2p^r}{a^2p^{2k}+b^2p^{2l}} &&&& \\
    & -\frac{2ap^{k+r}}{a^2p^{2k}+b^2p^{2l}} & -\frac{a^2p^{2k}-b^2p^{2l}}{a^2p^{2k}+b^2p^{2l}} & -\frac{2abp^{k+l}}{a^2p^{2k}+b^2p^{2l}} && \\
    & \frac{2bp^{l+r}}{a^2p^{2k}+b^2p^{2l}} & \frac{2abp^{k+l}}{a^2p^{2k}+b^2p^{2l}} & -\frac{a^2p^{2k}-b^2p^{2l}}{a^2p^{2k}+b^2p^{2l}} && \\
    &&&& 1 & \\
    & -p^r & -ap^k & -bp^l && \frac{a^2p^{2k}+b^2p^{2l}}{2p^r}
\end{psmallmatrix} \\
&=\begin{psmallmatrix}
    1 &&&&& \\
    & 1 & -\frac{2ap^k}{a^2p^{2k}+b^2p^{2l}} & \frac{2bp^l}{a^2p^{2k}+b^2p^{2l}} && -\frac{2}{a^2p^{2k}+b^2p^{2l}} \\
    && 1 &&& \frac{2ap^k}{a^2p^{2k}+b^2p^{2l}} \\
    &&& 1 && -\frac{2bp^l}{a^2p^{2k}+b^2p^{2l}} \\
    &&&& 1 & \\
    &&&&& 1
\end{psmallmatrix}\begin{psmallmatrix}
    1 &&&&& \\
    &&&&& 1 \\
    && 1 &&& \\
    &&& -1 && \\
    &&&& 1 & \\
    & 1 &&&&
\end{psmallmatrix}\begin{psmallmatrix}
    1 &&&&& \\
    & -p^r &&&& \\
    && 1 &&& \\
    &&& 1 && \\
    &&&& 1 & \\
    &&&&& -\frac{1}{p^r}
\end{psmallmatrix}\begin{psmallmatrix}
    1 &&&&& \\
    & 1 & \frac{ap^k}{p^r} & \frac{bp^l}{p^r} && -\frac{a^2p^{2k}+b^2p^{2l}}{2p^{2r}} \\
    && 1 &&& -\frac{ap^k}{p^r} \\
    &&& 1 && -\frac{bp^l}{p^r} \\
    &&&& 1 & \\
    &&&&& 1
\end{psmallmatrix}.
\end{align*}

\noindent We conclude that 
$$n_{s_\alpha,r}\coloneqq\begin{psmallmatrix}
    1 &&&&& \\
    &&&&& -p^{-r} \\
    && 1 &&& \\
    &&& -1 && \\
    &&&& 1 & \\
    & -p^r &&&&
\end{psmallmatrix}$$
for some $r\in\Z_{\geq0}$. For $p\equiv3\,(\mo 4)$, $r$ is even and we obtain the Kloosterman sum
\begin{align}
\kl_p(\psi,\psi';n_{s_\alpha,r}) &= \sum\limits_{\substack{a,b\,(\mo p^r) \\ (a,b,p^{r})=p^{r/2}}}e\left(-2\frac{m_2a-m_3b}{a^2+b^2}\right)e\left(\frac{n_2a+n_3b}{p^r}\right) \\
&=\sum\limits_{\substack{a,b\,(\mo p^{r/2}) \\ (a,b,p)=1}}e\left(-2\frac{(m_2a-m_3b)(\overline{a^2+b^2})}{p^{r/2}}\right)e\left(\frac{n_2a+n_3b}{p^{r/2}}\right). \label{eq:so42_p_odd}
\end{align}

\noindent For $p=2$, the terms are the same as for $p\equiv3\,(\mo 4)$, but the values of $a,b\,(\mo p^r)$ that are being summed over depend on the parity of $r$:
\begin{align}
\kl_p(\psi,\psi';n_{s_\alpha,r}) &= \begin{cases}
\sum\limits_{\substack{a,b\,(\mo p^r) \\ (a,b,p^r)=p^{(r+1)/2} \\ p^{(r+3)/2}\vert a\text{ or }p^{(r+3)/2}\vert b}}e\left(-p\frac{m_2a-m_3b}{a^2+b^2}\right)e\left(\frac{n_2a+n_3b}{p^r}\right)    &\text{ if } \text{$r$ is odd, $r>1$}, \\ 
& \\
\sum\limits_{\substack{a,b\,(\mo p^r) \\ (a,p^r)=(b,p^r)=p^{r/2}}}e\left(-p\frac{m_2a-m_3b}{a^2+b^2}\right)e\left(\frac{n_2a+n_3b}{p^r}\right)    &\text{ if } \text{$r$ is even}
\end{cases} \nonumber \\
&= \begin{cases}
\sum\limits_{\substack{a,b\,(\mo p^{(r-1)/2}) \\ (a,b,p)=1 \\ p\vert ab}}e\left(-\frac{(m_2a-m_3b)(\overline{a^2+b^2})}{p^{(r-1)/2}}\right)e\left(\frac{n_2a+n_3b}{p^{(r-1)/2}}\right) &\text{ if $r$ is odd, $r>1$}, \\ 
& \\
\sum\limits_{\substack{a,b\,(\mo p^{r/2}) \\ (ab,p)=1}}e\left(-\frac{(m_2a-m_3b)\overline{(a^2+b^2)p^{-1}}}{p^{r/2}}\right)e\left(\frac{n_2a+n_3b}{p^{r/2}}\right) &\text{ if $r$ is even}
\end{cases}. \label{eq:so42_p=2}
\end{align}
Note that the conditions on $a,b$ in the first sum ensure that $a^2+b^2$ is invertible modulo $p^{(r-1)/2}$, whereas the conditions on $a,b$ in the second sum ensure that $(a^2+b^2)p^{-1}$ is invertible modulo $p^{r/2}$.\newline

\noindent If $p=2,r=1$, then $R_w=\emptyset$, so the resulting Kloosterman sum is $0$.
\end{enumerate}

\subsection{Bounds}\label{sec:so42_bounds}
The only non-trivial and non-classical $SO_{4,2}$ Kloosterman sum obtained in the previous section is the one attached to $w=s_\alpha$. We denote the sum by $S_4(m_2,m_3,n_2,n_3;p^r)$ as in \eqref{eq:so42_sum}.

A trivial bound for $S_4(m_2,m_3,n_2,n_3;p^r)$ is given by $p^r$. It is not too difficult to obtain the non-trivial bound $p^{3r/4}$ if $(m_2m_3n_2n_3,p)=1$: one splits the sum according to whether $a$ or $b$ is a unit and then one applies Weil's bound. Such a strategy however does not optimally account for cancellation among the inner sums into which $S_4$ is decomposed. To make up for this, we apply a $p$-adic stationary phase method and the following result of Hooley, which is a consequence of Deligne's theorem on the Riemann hypothesis for $L$-functions of algebraic varieties over finite fields.

\begin{theorem}[{\cite[Theorem 5]{H}}]\label{thm:hooley}
Let $f(x_1,x_2,x_3),g(x_1,x_2,x_3)\in\F_p[x_1,x_2,x_3]$ be such that the variety $V=(g)$ is absolutely irreducible. For each $t\in \F_p$, assume that the variety $W(t)$ defined by
$$\begin{cases}
f(x_1,x_2,x_3)-t &= 0 \\
g(x_1,x_2,x_3) &= 0
\end{cases}$$
is such that 
\begin{enumerate}
\item generically $W(t)$ is an absolutely irreducible curve,
\item $W(t)$ is a curve (possibly reducible) or the zero variety for all specializations of $t$ to $\overline{\F}_p$.
\end{enumerate}
Then 
$$\sum_{\substack{x_1,x_2,x_3\,(\mo p) \\ g(x_1,x_2,x_3)\equiv0\,(\mo p)}}e\left(\frac{f(x_1,x_2,x_3)}{p}\right)=O(p),$$
where the implicit constant depends at most on $\deg(f)$ and $\deg(g)$.
\end{theorem}

\begin{proof}[Proof of Theorem \ref{thm:so42_bounds}]
Consider first the case $p\equiv3\,(\mo 4)$ and $r=2$, in which case we apply Theorem \ref{thm:hooley} above. Following the notation of the theorem, we take $x_1=a$, $x_2=b$, $x_3\equiv\overline{a^2+b^2}\,(\mo p)$, $f(x_1,x_2,x_3)=-2(m_2x_1-m_3x_2)x_3+n_2x_1+n_3x_2$, and $g(x_1,x_2,x_3)=(x_1^2+x_2^2)x_3-1$. Clearly $V=(g)$ is irreducible. We now show that the other necessary condition holds. Define $K$ to be the algebraic closure of $\F_p[t]$. Then the affine scheme $W(t)$ corresponds to the coordinate ring
$$\frac{K[x_1,x_2,x_3]}{(f-t,g)} = \frac{K[x_1,x_2,\overline{x_1^2+x_2^2}]}{(h-t)},$$
where $h$ results from substituting $x_3=\overline{x_1^2+x_2^2}$ in $f$. To see that $W(t)$ is irreducible, it suffices to show that the coordinate ring is an integral domain, i.e. we need to check that $r(x_1,x_2)\coloneqq(x_1^2+x_2^2)(n_2x_1+n_3x_2-t)-2(m_2x_1-m_3x_2)\in K[x_1,x_2]$ is irreducible. First of all, assume that $(m_2,m_3,p)=(n_2,n_3,p)=1$. Write $r=R_1+R_2+R_3$ with $R_i$ homogeneous in $x_1,x_2$ of degree $i$. If $r=st$ for non-units $s,t$, then we can have $s=S_0+S_1,t=T_1+T_2$ or $s=S_1,t=T_0+T_1+T_2$. The first case implies that $S_1$ shares a factor with $T_2$ and $x_1^2+x_2^2$, which is impossible because $S_1T_2=(x_1^2+x_2^2)(n_2x_1+n_3x_2)$ has no repeated factors. In the second case, we can assume wlog. that $s=m_2x_1-m_3x_2$ and then the only problem arises when $t=0$ and $n_2x_1+n_3x_2$ is proportional to $s$. To prevent this from happening, we assume that $(m_2n_3+m_3n_2,p)=1$.

For $p\equiv3\,(\mo 4)$ and $r>2$, we use the congruence relation
\begin{align*}
&\overline{(a+p^tx)^2+(b+p^ty)^2} \equiv \overline{(a^2+b^2)+2(ax+by)p^t+(x^2+y^2)p^{2t}} \,(\mo p^{3t}) \\
&\equiv \overline{a^2+b^2} - (2(ax+by)+(x^2+y^2)p^t)(\overline{a^2+b^2})^2p^t + (2(ax+by)+(x^2+y^2)p^t)^2(\overline{a^2+b^2})^3p^{2t} \,(\mo p^{3t}) \\
&\equiv \overline{a^2+b^2} - (2(ax+by)+(x^2+y^2)p^t)(\overline{a^2+b^2})^2p^t + 4(ax+by)^2(\overline{a^2+b^2})^3p^{2t} \,(\mo p^{3t})
\end{align*}
for any $t\in\Z_{\geq1}$, $a,b\in\Z_p^\times$ and $x,y\in\Z_p$. If $r=4t$ for some $t\in\Z_{\geq1}$, then
\begin{align}
&S_4(m_2,m_3,n_2,n_3;p^r) \\
&= \sum_{\substack{a,b\,(\mo p^t) \\ (a,b,p)=1}}\sum_{x=0}^{p^t-1}\sum_{y=0}^{p^t-1}e\left(\frac{-2(m_2(a+p^tx)-m_3(b+p^ty))(\overline{a^2+b^2})(1-2(ax+by)(\overline{a^2+b^2})p^t)}{p^{2t}}\right) \\
&\quad\quad\quad\,\,\,\,\,\,\quad\quad\quad\quad\quad\cdot e\left(\frac{n_2(a+p^tx)+n_3(b+p^ty)} {p^{2t}}\right) \\
&= \sum_{\substack{a,b\,(\mo p^t) \\ (a,b,p)=1}}e\left(\frac{-2(m_2a-m_3b)(\overline{a^2+b^2})+n_2a+n_3b}{p^{2t}}\right) \\
&\quad\quad\quad\quad\quad\,\, \cdot\sum_{x=0}^{p^t-1}e\left(\frac{(-2m_2(\overline{a^2+b^2})+4(m_2a-m_3b)a(\overline{a^2+b^2})^2+n_2)x}{p^t}\right) \label{eq:sum_of_the_form_1} \\
&\quad\quad\quad\quad\quad\,\, \cdot\sum_{y=0}^{p^t-1}e\left(\frac{(2m_3(\overline{a^2+b^2})+4(m_2a-m_3b)b(\overline{a^2+b^2})^2+n_3)y}{p^t}\right) \label{eq:sum_of_the_form_2} \\
&= p^{2t}\sum_{\substack{a,b\,(\mo p^{t}) \\ (a,b,p)=1 \\ f(a,b)\equiv g(a,b)\equiv0\,(\mo p^t)}}e\left(\frac{-2(m_2a-m_3b)(\overline{a^2+b^2}) + n_2a+n_3b}{p^{2t}}\right).
\end{align}
The last step uses the fact that the first inner sum \eqref{eq:sum_of_the_form_1} vanishes unless $a,b$ are such that
\begin{equation}\label{eq:f(a,b)}
f(a,b)\coloneqq -2m_2(\overline{a^2+b^2})+4(m_2a-m_3b)a(\overline{a^2+b^2})^2+n_2\equiv0\,(\mo p^t),
\end{equation}
and similarly the second inner sum \eqref{eq:sum_of_the_form_2} vanishes unless $a,b$ are such that
\begin{equation}\label{eq:g(a,b)}
g(a,b)\coloneqq2m_3(\overline{a^2+b^2})+4(m_2a-m_3b)b(\overline{a^2+b^2})^2+n_3\equiv0\,(\mo p^t).
\end{equation}
It thus suffices to estimate the number of $a,b\,(\mo p^t)$ with $(a,b,p)=1$ for which $f(a,b)\equiv g(a,b)\equiv0\,(\mo p^t)$. To answer this question, suppose that $(a,b)$ is such a solution with $(a,p)=1$, and define the new variable $c=a^{-1}b$. Then
\begin{align*}
\Tilde{f}(a,c) &\coloneqq f(a,b)(a^2+b^2)^2a^{-2} \equiv n_2(1+c^2)^2a^2 - 2(m_2c^2+2m_3c-m_2) \,(\mo p^t), \\
\Tilde{g}(a,c) &\coloneqq g(a,b)(a^2+b^2)^2a^{-2} \equiv n_3(1+c^2)^2a^2 -2(m_3c^2-2m_2c-m_3) \,(\mo p^t)
\end{align*}
are both linear in $a^2$, so $a^2$ can be uniquely expressed in terms of $c$ if $(n_2n_3,p)=1$. Solving $n_3\Tilde{f}(a,c)-n_2\Tilde{g}(a,c)\equiv -2(m_2n_3-m_3n_2)c^2-4(m_2n_2+m_3n_3)c+2(m_2n_3-m_3n_2)\equiv0\,(\mo p^t)$ for $c$ gives a bounded number of solutions unless the polynomial equation is completely degenerate. Each such solution $c$ yields at most $2$ solutions for $a$. We thus conclude that $S_4(m_2,m_3,n_2,n_3;p^r)\ll p^{2t}= p^{r/2}$. On the other hand, if $r=4t+2$ for some $t\in\Z_{\geq1}$, then similarly 
\begin{align*}
&S_4(m_2,m_3,n_2,n_3;p^r) \\
&= \sum_{\substack{a,b\,(\mo p^{t+1}) \\ (a,b,p)=1}}\sum_{x=0}^{p^t-1}\sum_{y=0}^{p^t-1}e\left(\frac{-2(m_2(a+p^{t+1}x)-m_3(b+p^{t+1}y))(\overline{a^2+b^2})(1-2(ax+by)(\overline{a^2+b^2})p^{t+1})}{p^{2t+1}}\right) \\
&\quad\quad\quad\quad\quad\quad\quad\quad\quad\quad\cdot e\left(\frac{n_2(a+p^{t+1}x)+n_3(b+p^{t+1}y)} {p^{2t+1}}\right) \\
&= p^{2t} \sum_{\substack{a,b\,(\mo p^{t+1}) \\ (a,b,p)=1 \\ f(a,b)\equiv g(a,b)\equiv 0\,(\mo p^t)}} e\left(\frac{-2(m_2a-m_3b)(\overline{a^2+b^2})+n_2a+n_3b}{p^{2t+1}}\right) \\
&= p^{2t}\sum_{\substack{a,b\,(\mo p^{t}) \\ (a,b,p)=1 \\ f(a,b)\equiv g(a,b)\equiv 0\,(\mo p^t)}} \sum_{x=0}^{p-1}\sum_{y=0}^{p-1} e\Bigg(\frac{-2(m_2(a+p^tx)-m_3(b+p^ty))(\overline{a^2+b^2})}{p^{2t+1}}\\
&\quad\quad\quad\quad\quad\quad\quad\quad\quad\quad\quad\quad\quad\quad\quad\quad\,\,\cdot (1-(2(ax+by)+(x^2+y^2)p^t)(\overline{a^2+b^2})p^t+4(ax+by)^2(\overline{a^2+b^2})^2p^{2t})\Bigg) \\
&\quad\quad\quad\quad\quad\quad\quad\quad\quad\quad\quad\quad\quad\,\quad\,\,\,\cdot e\left(\frac{n_2(a+p^tx)+n_3(b+p^ty)}{p^{2t+1}}\right) \\
&= p^{2t} \sum_{\substack{a,b\,(\mo p^t) \\ (a,b,p)=1 \\ f(a,b)\equiv g(a,b)\equiv0\,(\mo p^t)}} e\left(\frac{-2(m_2a-m_3b)(\overline{a^2+b^2})+n_2a+n_3b}{p^{2t+1}}\right) \\
&\quad\quad\quad\quad \cdot\sum_{x=0}^{p-1}\sum_{y=0}^{p-1} e\left(\frac{(6m_2a-2m_3b-8a^2(m_2a-m_3b)(\overline{a^2+b^2}))(\overline{a^2+b^2})x^2}{p}\right) \\
&\quad\quad\quad\quad\quad\quad\quad\,\,\,\,\,\cdot e\left(\frac{(4m_2b-4m_3a-16ab(m_2a-m_3b)(\overline{a^2+b^2}))(\overline{a^2+b^2})xy}{p}\right) \\
&\quad\quad\quad\quad\quad\quad\quad\,\,\,\,\,\cdot e\left(\frac{(2m_2a-6m_3b-8b^2(m_2a-m_3b)(\overline{a^2+b^2}))(\overline{a^2+b^2})y^2}{p}\right),
\end{align*}
with $f(a,b)$ and $g(a,b)$ as in \eqref{eq:f(a,b)} and \eqref{eq:g(a,b)}. \noindent As in the case $r=4t$, there are at most $\ll1$ solutions $(a,b)\in(\Z/p^t\Z)^{\oplus 2}$ to the congruence $f(a,b)\equiv g(a,b)\equiv0\,(\mo p^t)$ and such that $(a,b,p)=1$. The terms of the two inner sums involve a quadratic form $Ax^2+Bxy+Cy^2$, which can be diagonalized if and only if its discriminant $B^2-4AC\equiv16(m_2^2+m_3^2)(\overline{a^2+b^2})$ is non-zero modulo $p$. By \cite[Theorem 1.5.2]{BEW}, the two decoupled sums over $x$ and $y$ are then both of size at most $\sqrt{p}$. We conclude that $S_4(m_2,m_3,n_2,n_3;p^r)\ll p^{2t+1}=p^{r/2}$.

It remains to treat the case $p=2$. Suppose first that $r>1$ is odd. If $r=4t+1$ for some $t\in \Z_{\geq1}$, we obtain (as above) that
$$S_4(m_2,m_3,n_2,n_3;p^r) = p^{2t}\sum_{\substack{a,b\,(\mo p^t) \\ (a,b,p)=1 \\ p\vert ab \\ f(a,b)\equiv g(a,b)\equiv0\,(\mo p^t)}}e\left(\frac{-(m_2a-m_3b)(\overline{a^2+b^2})+n_2a+n_3b}{p^{2t}}\right),$$
where 
\begin{align}
&f(a,b) \coloneqq -m_2(\overline{a^2+b^2})+2(m_2a-m_3b)a(\overline{a^2+b^2})^2+n_2\,(\mo p^t), \label{eq:f_for_p=2_r_3mod4} \\
&g(a,b) \coloneqq m_3(\overline{a^2+b^2})+2(m_2a-m_3b)b(\overline{a^2+b^2})^2+n_3\,(\mo p^t).\label{eq:g_for_p=2_r_3mod4}
\end{align}
As before, suppose $(a,p)=1$ (so $p\vert b$) and write $c=a^{-1}b$. Then
\begin{align*}
\Tilde{f}(a,c)&\coloneqq f(a,b)(a^2+b^2)^2a^{-2}\equiv n_2(1+c^2)^2a^2-(m_2c^2+2m_3c-m_2)\,(\mo p^t), \\
\Tilde{g}(a,c)&\coloneqq g(a,b)(a^2+b^2)^2a^{-2}\equiv n_3(1+c^2)^2a^2-(m_3c^2-2m_2c-m_3)\,(\mo p^t)
\end{align*}
are both linear in $a^2$, so $a^2$ can be uniquely expressed in terms of $c$ if $(n_2n_3,p)=1$. Solving $n_3\Tilde{f}(a,c)-n_2\Tilde{g}(a,c)\equiv-(m_2n_3-m_3n_2)c^2-2(m_2n_2+m_3n_3)c+(m_2n_3-m_3n_2)\equiv0\,(\mo p^t)$ for $c$ gives a bounded number of solutions unless the polynomial equation is completely degenerate. Each such solution $c$ yields at most $4$ solutions for $a$. On the other hand, if $r=4t+3$ for some $t\in\Z_{\geq1}$, then similarly
$$S_4(m_2,m_3,n_2,n_3;p^r) = p^{2t}\sum_{\substack{a,b\,(\mo p^t) \\ (a,b,p)=1 \\ p\vert ab \\ f(a,b)\equiv g(a,b)\equiv0\,(\mo p^t)}}e\left(\frac{-(m_2a-m_3b)\overline{(a^2+b^2)}+n_2a+n_3b}{p^{2t+1}}\right)\sum_{x=0}^{p-1}\sum_{y=0}^{p-1}e\left(\cdots\right),$$
with $f(a,b)$ and $g(a,b)$ as in \eqref{eq:f_for_p=2_r_3mod4} and \eqref{eq:g_for_p=2_r_3mod4}. Since we are only considering the prime $p=2$ here, it suffices to trivially bound the two inner sums by $2^2$.

Next, we consider the case with $p=2$ and $r$ even. If $r=4t$ for some $t\in\Z_{\geq2}$, then 
\begin{equation}\label{eq:so42_r_even}
S_4(m_2,m_3,n_2,n_3;p^r) = p^{2t}\sum_{\substack{a,b\,(\mo p^t) \\ (ab,p)=1 \\ f(a,b)\equiv g(a,b)\equiv0\,(\mo p^t)}}e\left(\frac{-(m_2a-m_3b)\overline{(a^2+b^2)p^{-1}}+n_2a+n_3b}{p^{2t}}\right)
\end{equation}
if $(m_2m_3,p)=1$, where 
\begin{align}
&f(a,b) \coloneqq -m_2\overline{(a^2+b^2)p^{-1}}+(m_2a-m_3b)a\overline{(a^2+b^2)p^{-1}}^2+n_2\,(\mo p^t), \label{eq:f_for_p=2_r_even} \\
&g(a,b) \coloneqq m_3\overline{(a^2+b^2)p^{-1}}+(m_2a-m_3b)b\overline{(a^2+b^2)p^{-1}}^2+n_3\,(\mo p^t). \label{eq:g_for_p=2_r_even}
\end{align}
As before, write $c=a^{-1}b$. Then 
\begin{align*}
\Tilde{f}(a,c) &\coloneqq f(a,b)(a^2+b^2)^2p^{-2}a^{-2} \equiv n_2(1+c^2)^2p^{-2}a^2 - (m_2(1+c^2)p^{-1}+m_3c-m_2) \,(\mo p^t), \\
\Tilde{g}(a,c) &\coloneqq g(a,b)(a^2+b^2)^2p^{-2}a^{-2} \equiv n_3(1+c^2)^2p^{-2}a^2 + (m_3(1+c^2)p^{-1}-m_3c^2+m_2c)\,(\mo p^t)
\end{align*}
are both linear in $a^2$, so $a^2$ can be uniquely expressed in terms of $c$ if $(n_2n_3,p)=1$. Solving $(n_3\Tilde{f}(a,c)-n_2\Tilde{g}(a,c))p\equiv-(m_2n_3-m_3n_2)c^2-2(m_2n_2+m_3n_3)c+(m_2n_3-m_3n_2)\equiv0\,(\mo p^t)$ for $c$ gives a bounded number of solutions unless the polynomial equation is completely degenerate. Each such solution yields at most $4$ solutions for $a$. To obtain \eqref{eq:so42_r_even}, one applies the congruence relation
\begin{align*}
&\overline{((a+p^tx)^2+(b+p^ty)^2)p^{-1}} \\
&\equiv\overline{(a^2+b^2)p^{-1}}(1-((ax+by)+(x^2+y^2)p^{t-1})\overline{(a^2+b^2)p^{-1}}p^{t}+(ax+by)^2\overline{(a^2+b^2)p^{-1}}^2p^{2t})\,(\mo p^{2t+1})
\end{align*}
for any $t\in\Z_{\geq2}$, $a,b\in\Z_p^\times$ and $x,y\in\Z_p$. This results in two generalized quadratic Gauss sums (instead of just sums of the form \eqref{eq:sum_of_the_form_1} and \eqref{eq:sum_of_the_form_2}), which we know how to treat. On the other hand, if $r=4t+2$ for some $t\in\Z_{\geq2}$, then similarly
$$S_4(m_2,m_3,n_2,n_3;p^r) = p^{2t}\sum_{\substack{a,b\,(\mo p^t) \\ (ab,p)=1 \\ f(a,b)\equiv g(a,b)\equiv0\,(\mo p^t)}}e\left(\frac{-(m_2a-m_3b)\overline{(a^2+b^2)p^{-1}}+n_2a+n_3b}{p^{2t+1}}\right)\sum_{x=0}^{p-1}\sum_{y=0}^{p-1}e\left(\dotsc\right),$$
with $f(a,b)$ and $g(a,b)$ as in \eqref{eq:f_for_p=2_r_even} and \eqref{eq:g_for_p=2_r_even}. Again, we bound the two inner sums trivially by $2^2$.

For $p=2$, we excluded finitely many values of $r$ in order to be able to apply the $p$-adic stationary phase method. In the remaining cases, we apply the trivial bound $S_4(m_2,m_3,n_2,n_3;p^r)\ll1$.

\end{proof}

\section*{References}
\printbibliography[heading=none]

@incollection {B,
author = {Blomer, Valentin},
title = {The relative trace formula in analytic number theory},
booktitle = {Relative trace formulas},
series = {Simons Symposia},
pages = {51--73},
publisher = {Springer},
year = {2021},
MRCLASS = {11F72},
MRNUMBER = {4611943},
DOI = {10.1007/978-3-030-68506-5_2},
}

@book{BEW,
author = {Berndt, B. and Evans, R. and Williams, K.},
title = {Gauss and {J}acobi sums},
series = {Canadian Mathematical Society Series of Monographs and Advanced Texts},
note = {A Wiley-Interscience Publication},
publisher = {John Wiley \& Sons, Inc., New York},
year = {1998},
pages = {xii+583},
}

@article {BM,
author = {Blomer, Valentin and Man, S. H.},
title = {Bounds for Kloosterman sums on $GL(n)$},
journal = {Mathematische Annalen},
volume = {390},
year = {2024},
number = {1},
pages = {1171--1200},
doi = {10.1007/s00208-023-02777-6},
}

@article{BFG,
title = {Poincaré series and Kloosterman sums for $SL(3,\Z)$}, 
author = {Bump, D. and Friedberg, S. and Goldfeld, D.},
journal = {Acta Arithmetica},
volume = {50},
pages = {31-89},
year = {1988}
}

@misc{C,
author = {Conrad, K.},
title = {A multivariable Hensel's lemma}
}

@incollection {D,
author = {Deligne, P.},
title = {Applications de la formule des traces aux sommes trigonom\'etriques},
booktitle = {Cohomologie \'etale},
series = {Lecture Notes in Mathematics},
volume = {569},
pages = {168-232},
publisher = {Springer, Berlin},
year = {1977},
}

@article{DF,
title = {A stationary phase formula for exponential sums over $\Z/p\Z$ and applications to $GL(3)$-Kloosterman sums},
author = {D{\k a}browski, R. and Fisher, B.},
journal = {Acta Arithmetica},
volume = {80},
pages = {1-48},
year = {1997}
}

@article{DR,
title = {Kloosterman sets in reductive groups},
author = {D{\k a}browski, R. and Reeder, M.},
journal = {Journal of Number Theory},
volume = {73},
pages = {228-255},
year = {1998}
}

@article{GSW, 
title = {An orthogonality relation for $GL(4,\R)$ (with an appendix by Bingrong Huang)},
author = {Goldfeld, D. and Stade, E. and Woodbury, M.},
journal = {Forum of Mathematics, Sigma},
volume = {9},
pages = {e47},
year = {2021}
}

@incollection{H,
author = {Hooley, C.},
title = {On exponential sums and certain of their applications},
booktitle = {Number theory days, 1980 ({E}xeter, 1980)},
series = {NATO Advanced Study Institute Series C: Mathematical and Physical Sciences},
volume = {91},
pages = {92-122},
publisher = {Cambridge University Press},
year = {1982},
}

@article{jc,
  title={Poincaré series for $SO(n,1)$},
  author={Cogdell, J. and Piatetski-Shapiro, I. and Li, J. and Sarnak, P.},
  journal={Acta Mathematica},
  volume={167},
  pages={229-285},
  year={1991}
}

@article{je,
  title={Kloosterman sums for Clifford algebras and a lower bound for the positive eigenvalues of the Laplacian for congruence subgroups acting on hyperbolic spaces},
  author={Elstrodt, J. and Grunewald, F. and Mennicke, J.},
  journal={Inventiones Mathematicae},
  volume={101},
  pages={641-685},
  year={1990}
}

@article{S,
title = {Poincaré series on $GL(r)$ and Kloosterman sums},
author = {Stevens, G.},
journal = {Mathematische Annalen},
volume = {277},
pages = {25-52},
year = {1987}
}

@inbook{jsm2, 
title={Algebraic Groups: The Theory of Group Schemes of Finite Type over a Field}, 
author={Milne, J. S.},
year={2022}, 
publisher={Cambridge University Press},  
place={Cambridge}, 
series={Cambridge Studies in Advanced Mathematics}, 
collection={Cambridge Studies in Advanced Mathematics}
}

@article{K,
title={On the representation of numbers in the form $ax^2+by^2+cz^2+dt^2$},
author={Kloosterman, H.},
journal={Acta Mathematica},
volume={49},
pages={407-464},
year={1926}
}

@article{rd2,
  title={Kloosterman sums for Chevalley groups}, 
  author={D{\k a}browski, R.},
  journal={Transactions of the American Mathematical Society},
  volume={337},
  pages={757-769},
  year={1993}
}

@article{F,
title = {Poincaré series for $GL(n)$: Fourier expansion, Kloosterman sums, and algebreo-geometric estimates},
author = {Friedberg, S.},
journal = {Mathematische Zeitschrift},
volume = {196},
pages = {165-188},
year = {1987}
}

@article {M,
author = {Man, S. H.},
title = {Symplectic Kloosterman sums and Poincar\'e series},
journal = {Ramanujan Journal},
volume = {57},
year = {2022},
number = {2},
pages = {707--753},
DOI = {10.1007/s11139-021-00498-5},
}

@article{Sa,
author = {Sali\'e, H.},
title = {\"Uber die {K}loostermanschen {S}ummen {$S(u,v;q)$}},
journal = {Mathematische Zeitschrift},
volume = {34},
year = {1932},
number = {1},
pages = {91--109},
DOI = {10.1007/BF01180579},
}

@article{W,
title = {On some exponential sums},
author = {Weil, A.},
journal = {Proceedings of the National Academy of Sciences of the United States of America},
volume = {34},
year = {1948},
pages = {204--207},
doi = {10.1073/pnas.34.5.204}
}
\end{document}